\documentclass{cedram-aif} 

\usepackage[english]{babel}
\usepackage{hyperref}
\usepackage{amsmath}
\usepackage{amsthm}
\usepackage{amssymb}
\usepackage{ucs}
\usepackage{enumerate}
\usepackage[utf8x]{inputenc}
\usepackage[T1]{fontenc}
\usepackage{ae,aecompl}

\usepackage{color}
\usepackage{tikz}
\usepackage{diagbox}

\DeclareMathOperator{\rng}{\mathrm{rng}}
\DeclareMathOperator{\supp}{\mathrm{supp}}

\newcommand{\MAlg}{\mathrm{MAlg}}
\newcommand{\Aut}{\mathrm{Aut}}

  \newcommand{\R}{\mathbb R}
  \newcommand{\N}{\mathbb N}
  
  \newcommand{\Z}{\mathbb Z}
  
  \newcommand{\LL}{\mathrm L}
  \newcommand{\U}{\mathcal U}

 \newcommand{\dom}{\mathrm{dom}\;}

  \newcommand{\inv}{^{-1}}
  
  \renewcommand{\H}{\mathcal H}

  \renewcommand{\leq}{\leqslant}
  \renewcommand{\geq}{\geqslant}
  \newcommand{\abs}[1]{\left\lvert #1\right\rvert}
  \newcommand{\norm}[1]{\left\lVert #1\right\rVert}
  
  \newcommand{\act}{\curvearrowright}
  
  \newcommand{\impl}{\Rightarrow}

  \newcommand{\la}{\left\langle}
  \newcommand{\ra}{\right\rangle}

  \DeclareMathOperator{\diam}{\mathrm{diam}}

\newtheorem{thmi}{Theorem}

\theoremstyle{definition}

\newenvironment{exam}{\begin{enonce}[remark]{Example}}{\end{enonce}}
\newenvironment{ques}{\begin{enonce}{Question}}{\end{enonce}}
\newtheorem*{qui}{Question}
\newtheorem*{conji}{Conjecture}

\newtheorem*{prob}{Problem}
\newtheorem*{claim}{Claim}
\theoremstyle{remark}
\newtheorem*{ackn}{Aknowledgments}
\newenvironment{claimproof}
  {%
    \proof[Proof of claim]%
  }
  {%
    \endproof%
  }

\newcommand{\id}{\mathrm{id}}

\title[A measurable analogue of topological full groups II]{On a measurable 
analogue of small topological full groups II}

\author{François Le Maître}

\address{Université de Paris, Sorbonne Université, CNRS,\\ 
	Institut de Mathématiques de Jussieu-Paris Rive Gauche,\\
	F-75013 Paris, France}
\begin{document}

\maketitle

\begin{abstract}
We pursue the study of $\LL^1$ full groups of graphings and of the closures of their derived groups, which we call derived $\LL^1$ full groups. Our main result shows that aperiodic probability measure-preserving actions of finitely generated groups have finite Rokhlin entropy if and only if their derived $\LL^1$ full group has finite topological rank.  We further show that a graphing is amenable if and only if its $\LL^1$ full group is, and explain why various examples of (derived) $\LL^1$ full groups fit very well into Rosendal's geometric framework for Polish groups. As an application, we obtain that every abstract group isomorphism between $\LL^1$ full groups of amenable ergodic graphings must be a quasi-isometry for their respective $\LL^1$ metrics. We finally show that $\LL^1$ full groups of rank one transformations have topological rank 2. 
\end{abstract}

\section{Introduction}

In his pioneering work \cite{dyeGroupsMeasurePreserving1959}, Dye defined \textbf{full} subgroups of the group $\Aut(X,\mu)$ of all measure-preserving transformations of a standard probability space $(X,\mu)$ as those subgroups $\mathbb G$ which are stable under cutting and pasting their elements along a countable measurable partition: for every measurable partition $(A_n)_{n\in\N}$  of $X$ and every sequence $(T_n)_{\in\N}$ of elements of $\mathbb G$ such that $(T_n(A_n))_{n\in\N}$ is also a partition of $X$, the measure preserving transformation $T$ defined by $T(x)=T_n(x)$ for all $x\in A_n$ also belongs to $\mathbb G$. 
Then, given a measure-preserving action of a countable group $\Gamma$ on a standard probability space $(X,\mu)$ (i.e. a homomorphism $\alpha:\Gamma\to\Aut(X,\mu)$), the smallest full subgroup of $\Aut(X,\mu)$ containing $\alpha(\Gamma)$ is called the \textbf{full group} $[\alpha(\Gamma)]$ of the action $\alpha$. 

 The full group of the action $\alpha$ can be concretely understood as the group of all measure-preserving transformations $T$ such that $T(x)$ belongs to the $\Gamma$-orbit of $x$ for all $x\in X$. Moreover, if two actions are conjugate, then their full groups are isomorphic. It was then realized that the full group was exactly capturing the partition of the space into orbits induced by the action, leading to the notion of \emph{orbit equivalence} which is by now a fairly large subject on its own (see for instance the survey \cite{gaboriauOrbitEquivalenceMeasured2011}).

The present paper is about a family of smaller subgroups which one can associated to any measure-preserving action of any \emph{finitely generated} group. These were defined in \cite{lemaitremeasurableanaloguesmall2018} and are called $\LL^p$ full groups, although they are not full as subgroups of $\Aut(X,\mu)$, but only \textbf{finitely full}, which means that they are stable under cutting and pasting their elements along a \emph{finite} measurable partition.

 Their definition uses the extra geometric structure provided by the Schreier graph of the action. To be more precise, given a measure preserving action $\alpha:\Gamma\to\Aut(X,\mu)$ of a finitely generated group $\Gamma$ and a finite symmetric generating set $S$ for $\Gamma$, let us first denote by $d_S$ the path metric associated to the Schreier graph of the action with respect to the set $S$. For $p\in[1,+\infty[$, we then define the $\LL^p$ \textbf{full group} of $(\alpha,S)$ be the group of all $T\in [\alpha(\Gamma)]$ such that the map 
$$x\mapsto d_S(x,T(x))$$
is in $\LL^p(X,\mu)$. Changing the generating set $S$ does not change the $\LL^p$ full group, so the $\LL^p$ full group of the action $\alpha$ itself is well-defined, and we denote it by $[\alpha(\Gamma)]_p$. If two actions are conjugate, then their $\LL^p$ full groups are isomorphic. Moreover, the natural $\LL^p$ metric on $\LL^p$ full groups is complete, separable, right-invariant and endows the $\LL^p$ full groups with a group topology which does not depend on the choice of the finite generating set. It is thus natural to ask the following.
\begin{qui}
	 How are properties of the action reflected in the $\LL^p$ full group as a 
	 topological group ?	 
 \end{qui} 
One of the natural invariants of topological groups is the \textbf{topological rank}, namely the smallest number of elements needed to generated a dense subgroup, so it is natural to ask how the action influences the topological rank of the associated $\LL^p$ full group. \\

 In \cite{lemaitremeasurableanaloguesmall2018} we showed that the topological rank of the $\LL^1$ full group of an ergodic $\Z$-action was finite if and only if the action had finite entropy.
The main goal of the present paper is to extend this result to general 
measure-preserving actions of finitely generated groups. In order to do so, the 
right notion of entropy is \emph{Rokhlin entropy}, which was recently 
introduced by Seward \cite{sewardKriegerfinitegenerator2019}. It is the infimum 
of the entropies of generating partitions and, for $\Z$-actions, it coincides 
with the usual definition of entropy of a measure-preserving transformation by 
a result of Rokhlin \cite{rokhlinLecturesentropytheory1967}. A crucial result 
of Seward is that the finiteness of the Rokhlin entropy characterizes 
measure-preserving subshifts, generalizing Krieger's theorem to the widest 
imaginable setup (see Section \ref{sec: def Rokhlin entropy} for more on this).

Our proof that ergodic $\Z$-actions of finite entropy have a topologically finitely generated $\LL^1$ full group used crucially the \textbf{derived $\LL^1$ full group} , defined as the closure of the  subgroup generated by commutators in the $\LL^1$ full group of the action\footnote{The quotient of the $\LL^1$ full group by the derived $\LL^1$ full group is the largest abelian continuous quotient of the $\LL^1$ full group, and is thus called the \emph{topological abelianization} of the $\LL^1$ full group.}. Indeed, the latter is far easier to understand, in particular it is topologically generated by naturally defined involutions arising from the generating set of $\Gamma$ (see Theorem \ref{thm: topo gen by I_phi}  below). The main result from \cite{lemaitremeasurableanaloguesmall2018} is then a direct consequence of the following two statements:

\begin{enumerate}[(1)]
	\item The derived $\LL^1$ full group is topologically finitely generated as soon as the action has finite entropy.
	\item The topological abelianization of the $\LL^1$ full group of an ergodic $\Z$-action is isomorphic to $\Z$, in particular it always has finite topological rank. 
\end{enumerate}

 In the present paper, we generalize the first item from the above list as follows (see Section \ref{sec: proof of finite topo rank} for the proof). 

\begin{thmi}\label{thmi: finite topo rank}
	Let $\Gamma$ be a finitely generated group, let $\alpha:\Gamma\to\Aut(X,\mu)$ be a measure-preserving aperiodic $\Gamma$-action on $(X,\mu)$. The following are equivalent.
	\begin{enumerate}[(i)]
		\item The action has finite Rokhlin entropy.
		\item \label{itemi: finite topo rank}The derived $\LL^1$ full group of the action has finite topological rank. 
	\end{enumerate}
\end{thmi}

Apart from Seward's characterization of subshifts, another key ingredient in the proof is Nekrashevych's work on alternating topological full groups \cite{nekrashevychSimplegroupsdynamical2017}, which provides dense finitely generated subgroups in the derived $\LL^1$ full groups. The reader is referred to Section \ref{sec: topo full def} for the definition of alternating topological full group and Section \ref{sec: dense topo full} for the proof of their density in $\LL^1$ full groups.

In the above result, it is natural to ask whether one can remove the “derived” adjective in item \eqref{itemi: finite topo rank}, as in the $\Z$-case. However, we have no idea what the topological abelianization of the $\LL^1$ full group of a general measure-preserving ergodic action of a finitely generated group looks like, and in particular we do not know when it happens to be topologically finitely generated. We don't have any example or non example, even for $\Z^2$-actions. \\

Let us now consider another invariant of topological groups: amenability. In \cite{lemaitremeasurableanaloguesmall2018} we showed that if the $\LL^1$ full group of a free action of a finitely generated group $\Gamma$ was amenable, then $\Gamma$ itself was amenable, and left the question of the converse open. Here we show that the converse holds so as to obtain the following result (see section \ref{sec:amenability} for definitions and a more general statement in terms of \emph{graphings}).

\begin{thmi}\label{thmi: amenable characterization}
Let $\Gamma$ be a finitely generated group acting freely on a standard probability space $(X,\mu)$. Then the following are equivalent:
\begin{enumerate}[(i)]
	\item $\Gamma$ is amenable;
	\item The $\LL^1$ full group of the action is amenable;
	\item The derived $\LL^1$ full group of the action is amenable;
	\item The derived $\LL^1$ full group of the action is extremely amenable;
	\item The derived $\LL^1$ full group of the action is whirly amenable.
\end{enumerate}
\end{thmi}

Note that the whole $\LL^1$ full group of any free ergodic $\Z$-action factors onto $\Z$ \cite[Corollary 4.20]{lemaitremeasurableanaloguesmall2018} and hence cannot be extremely amenable. So we cannot have a stronger notion of amenability for the whole $\LL^1$ full group in the above equivalences.\\

We also note that $\LL^1$ full groups provide a rich playground for Rosendal's extension of geometric group theory to Polish groups: the $\LL^1$ metric can often canonically be obtained (up to quasi-isometry) purely from the topological group structure. We refer the reader to Section \ref{sec: geometry} for precise statements, and single out the following application, proved in Section \ref{sec: automatic QI}.

\begin{thmi}\label{thmi: auto QI}
	Let $\Gamma$ and $\Lambda$ be two finitely generated amenable groups acting in a measure-preserving and ergodic manner on $(X,\mu)$. Then every abstract group isomorphism between their $\LL^1$ full groups must be a quasi-isometry for their respective $\LL^1$ metrics. 
\end{thmi}

The conclusion of the theorem is somehow optimal: if we replace quasi-isometry by bilipschitz equivalence, the theorem fails (see Example \ref{ex: not bilip iso}). We don't know how to generalize Theorem \ref{thmi: auto QI} to the nonamenable setup, but we conjecture that the following stronger statement holds.  

\begin{conji}
	Let $\Phi$ be a graphing. Then the $\LL^1$ metric on the $\LL^1$ full group of $\Phi$ is maximal. 
\end{conji}

Moreover, the following natural problem arises.

\begin{prob}
	Classify (derived) $\LL^1$ full groups up to quasi-isometry.
\end{prob}

A simpler instance of the above problem is to decide whether the derived $\LL^1$ full groups of the $2$-odometer and of the $3$-odometer are quasi-isometric, which has a natural reformulation in terms of quasi-isometry of the dyadic and triadic symmetric groups for an $\LL^1$-like metric (see the paragraph following Question \ref{qu: quasi-isom 2 3 odometer} for details).\\

Finally, we make progress on the computation of the exact topological rank of $\LL^1$ full groups of measure-preserving ergodic transformations. As reminded before, we had shown in \cite{lemaitremeasurableanaloguesmall2018} that the topological rank of $\LL^1$ full groups of a measure-preserving ergodic transformations is finite if and only if the transformation has finite entropy. Using Marks' work \cite{marksTopologicalGeneratorsFull2016}, we moreover had obtained that the topological rank of the $\LL^1$ full group of irrational rotations is equal to two. 

Here, we extend this result to the class of \emph{rank one} transformations, which we will properly define in section \ref{sec: rank one}. Let us for now mention that every compact transformation is rank one (in particular every irrational rotation is rank one), and that every rank one transformation has entropy zero.

\begin{thmi}
	Let $T$ be a rank-one measure-preserving transformation. Then the topological rank of its $\LL^1$ full group is equal to $2$.
\end{thmi}

It would be very interesting to have an example of a measure-preserving ergodic transformation such that the topological rank of its $\LL^1$ full group is finite, but greater than $2$. \\

The paper ends with further remarks and questions related to these results. We introduce the notion of $\LL^p$-full orbit equivalence and ask whether it is equivalent to $\LL^p$ orbit equivalence, we mention a natural extension of Theorem \ref{thmi: finite topo rank} to \emph{symmetric} $\LL^p$ full groups. We finally make a few more remarks on the topological abelianization of $\LL^1$ full groups.

\begin{ackn}
	I am very grateful to Alessandro Carderi, Nicolás Matte Bon and Todor 
	Tsankov for many enriching conversations around the topics of the present 
	paper. I would also like to thank Sébastien Gouëzel for allowing me to 
	include Theorem \ref{thm: not linf equivalent} in this paper, and the 
	anonymous 
	referee for their helpful remarks, especially for suggesting that the 
	answer to Question 5.5 from the first version of the paper should be 
	negative.\\
	\indent Research supported by Projet ANR-14-CE25-0004 GAMME and Projet 
	ANR-17-CE40-0026 AGRUME.

\end{ackn}

\tableofcontents

\section{Preliminaries}

In what follows $\Gamma$ and $\Lambda$ will always denote countable \emph{infinite} groups. 

\subsection{Standard probability spaces}

Recall that a measurable space is \textbf{standard} if its $\sigma$-algebra comes from the Borel $\sigma$-algebra induced by the topology of some Polish space.  
Measurable maps between standard Borel spaces are simply called Borel maps. One of the highlights of standard Borel spaces is the fact that the range of any injective Borel map  is Borel so that Borel bijections automatically have a Borel inverse \cite[Theorem 15.1]{kechrisClassicaldescriptiveset1995}. Another important fact is that in any standard Borel space we have a countable family of Borel subsets $(C_n)$ which \textbf{separates points}: for every distinct $x,y\in X$ there is $n\in\N$ such that $x\in C_n \nLeftrightarrow y\in C_n$.

Given a countable group $\Gamma$, a $\Gamma$-action on $X$ is called Borel if the bijections induced by the elements of $\Gamma$ are Borel maps. 
Two Borel actions $\Gamma\act X$ and $\Gamma\act Y$ are \textbf{conjugate} if there is a Borel $\Gamma$-equivariant bijection between $X$ and $Y$.

The \textbf{support} of a Borel bijection $T:X\to X$ is the Borel set $\supp T:=\{x\in X: T(x)\neq x\}$.
Here are two useful lemmas on Borel bijections of standard Borel spaces which only rely on the existence of a countable family of Borel subsets which separates points.

\begin{lemm}[{see e.g. \cite[Proposition 
2.7]{lemaitremeasurableanaloguesmall2018}}]
Let $T$ be a Borel bijection of a standard Borel space $X$. Then there is a partition of $\supp T$ in three Borel subsets  $A_1, A_2, A_3$ such that for all $i\in \left\{1,2,3\right\}$ the set $T(A_i)$ disjoint from $A_i$. 
\end{lemm}
%

By applying the above lemma $n$ times we have.

\begin{lemm} \label{lem: partitioning intersection of supports}
Let $X$ be a standard Borel space, let $(T_i)_{i=1}^n$ be a finite family of Borel bijections of $X$. 
Then there is a finite Borel partition $(A_j)_{j=1}^m$ of $\bigcap_{i=1}^n\supp T_i$ such that for each $i\in \left\{1,...,n\right\}$ and each $j\in \left\{1,...,m\right\}$ we have that $T_i(A_j)$ is disjoint from $A_j$. \qed
\end{lemm}
 Recall that given a measurable map $f:X\to Y$ between measurable spaces, if $\mu$ is a probability measure on $X$ then one can define the \textbf{pushforward measure} $f_*\mu$ on $Y$ by $f_*\mu(A)=\mu(f\inv(A))$ for all every $A\subseteq Y$.

We will  work most of the time in  a \textbf{standard probability space} $(X,\mu)$, which means that $X$ is still a standard Borel space but now equipped with a nonatomic Borel probability measure $\mu$. Given two standard probability spaces $(X, \mu)$ and $(Y,\nu)$, there is a Borel bijection $S: X\to Y$ such that $S_*\mu=\nu$ \cite[Theorem 17.41]{kechrisClassicaldescriptiveset1995}, so we may as well work in a fixed standard probability space $(X,\mu)$.

\subsection{The measure algebra and its automorphisms}\label{sec: meas alg and 
auto}

Let $(X,\mu)$ be a standard probability space, we denote by $\MAlg(X,\mu)$ the 
quotient of the algebra $\mathcal B(X)$ of Borel subsets of $X$ by the ideal of 
measure zero sets; in other words in the measure algebra we identify two Borel 
sets $A$ and $B$ as soon as $\mu(A\bigtriangleup B)=0$. Unless specified 
otherwise, we will always be working in the measure algebra when considering 
Borel sets, so for instance when we write $A\subseteq B$ we mean 
$\mu(A\setminus B)=0$. 

A \textbf{measure-preserving transformation} of $(X,\mu)$ is a Borel bijection $T:X\to X$ such that $T_*\mu=\mu$. 
We denote by $\Aut(X,\mu)$ be the group of all measure-preserving transformations, two such transformations being identified if they coincide on a full measure subset of $X$. We say that a measure-preserving transformation is \textbf{periodic} if all its orbits are finite.

 Every element of $\Aut(X,\mu)$ defines a measure-preserving automorphism of the measure algebra, and conversely every automorphism of the measure algebra comes from a measure-preserving transformation which is unique up to measure zero \cite[Theorem 15.11]{kechrisClassicaldescriptiveset1995}. In particular, since every measure-preserving Borel bijection between full measure Borel subsets of $(X,\mu)$ defines an automorphism of its measure algebra, it also defines a measure-preserving transformation. This will be used implicitly in many proofs. 
 
 Given $T\in\Aut(X,\mu)$ and $A\subseteq X$ of positive measure, we define the $T$-\textbf{return time} to $A$ as the function $n_{T,A}:X\to\N\setminus\{0\}$ given by: for all $x\in A$, $n_{A,T}(x)$ is the smallest $n\in\N$ such that $T^n(x)\in A$. The Poincaré recurrence theorem ensures us that the return time is well defined up to measure zero, also called the \textbf{return time} to $A$. We can then define the measure-preserving transformation $T_A$ \textbf{induced} by $T$ on $A$ by: 
 \[
 T_A(x)=\left\{\begin{array}{cc}T^{n_{A,T}}(x) & \text{if }x\in A \\x & \text{else.}\end{array}\right.
 \]

The  measure algebra $\MAlg(X,\mu)$ can be equipped with a complete metric $d_\mu$ defined by $d_\mu(A,B)=\mu(A\bigtriangleup B)$.
The following lemma which was used implicitly in \cite{lemaitremeasurableanaloguesmall2018}.

\begin{lemm}\label{lem: separating and density}
	Let $(X,\mu)$ be a standard probability space, let $\mathcal D$ be a set of Borel subsets of $X$ such that the image of $\mathcal D$ in $\MAlg(X,\mu)$ is dense. Then there is a full measure Borel subset $X'\subseteq X$ such that $\mathcal D$ separates the points of $X'$.
\end{lemm} 
\begin{proof}
	Let $(C_n)_{n\in\N}$ separate the points of $X$, by changing our enumeration we may as well assume each $C_n$ appears infinitely often. Now for each $n\in\N$ we pick $D_n\in\mathcal D$ such that $\mu(C_n\bigtriangleup D_n)<2^{-n}$. By the Borel-Cantelli lemma, there is a full measure Borel subset $X'\subseteq X$ such that for all $x\in X'$, there is $N\in\N$ such that for all $n\geq N$, $x\not\in C_n\bigtriangleup D_n$. By construction the set $X'$ is as wanted. 
\end{proof}

\subsection{Ergodic theory}

An action of a countable group $\Gamma$ on a standard probability space $(X,\mu)$ is \textbf{measure-preserving} if it is a Borel action and for every $\gamma\in \Gamma$ we have $\gamma_*\mu=\mu$. Every such action yields a homomorphism $\Gamma\to \Aut(X,\mu)$ which only depends on the action up to measure zero and conversely, every such homomorphism comes from a measure-preserving action which is unique up to measure zero (see \cite[Theorem 2.15]{glasnerErgodicTheoryJoinings2003}). 

Two measure-preserving actions of $\Gamma$ are \textbf{conjugate} if they are conjugate as Borel actions up to restricting to full measure $\Gamma$-invariant Borel subsets. Viewing actions as homomorphisms $\alpha,\beta: \Gamma\to\Aut(X,\mu)$, this is equivalent to the existence of $T\in\Aut(X,\mu)$ such that for all $\gamma\in\Gamma$ we have $\beta(\gamma)=T\alpha(\gamma)T\inv$.
The $\LL^1$ full groups that we are interested in only depend on the action up to measure zero, so we can and will forget about what happens on measure zero sets most of the time. By doing so we will implicitly use the following : if we have a Borel set $A$ of full measure, then it contains the full measure $\Gamma$-invariant set $\bigcap_{\gamma\in \Gamma}\gamma A$.

Two basic invariants of conjugacy of measure preserving actions are \textbf{aperiodicity} (the $\Gamma$-orbit of almost every $x\in X$ is infinite) and the stronger notion of \textbf{ergodicity} (every $\Gamma$-invariant Borel subset has measure $0$ or $1$). We will only be interested in aperiodic actions, and sometimes we will restrict to ergodic ones. Another important invariant is \textbf{freeness}, which means that for almost all $x\in X$ and all $\gamma\in \Gamma$ we have $\gamma x\neq x$. Note that aperiodicity follows from freeness since all the groups we consider are infinite.

\begin{exam}Let $I$ be a countable set and let $\nu$ be a probability measure 
on $I$ which is not a Dirac measure, then the measure preserving action 
$\Gamma\act (I^\Gamma,\nu^{\otimes \Gamma})$ given by $(\gamma \cdot 
x)_g=x_{\gamma\inv g}$ is free and ergodic. Such an action is called a 
\textbf{Bernoulli shift}.
\end{exam}

\subsection{Subshifts and Rokhlin entropy} \label{sec: def Rokhlin entropy}

Let $X$ be a standard Borel space, let $\Gamma\act X$ be a Borel action. An \textbf{observable} on $X$ is a Borel map $\varphi: X\to I$ where $I$ is a countable set. To every such observable we can associate a Borel partition of $X$ indexed by $I$ defined as $(\varphi\inv\{i\})_{i\in I}$, and conversely if $\mathcal P$ is a countable partition of $X$ then the map $\varphi: X\to\mathcal P$ which sends every $x\in X$ to the unique $P\in\mathcal P$ such that $x\in P$ is an observable. 
An observable $\varphi$  is called  (dynamically) \textbf{generating} if for all $x,y\in X$ with $x\neq y$ there is $\gamma\in \Gamma$ such that $\varphi(\gamma x)\neq\varphi(\gamma y)$. 

A Borel action $\Gamma\act X$ on a standard Borel space is called a \textbf{subshift} if it admits a generating observable with \emph{finite} range. 
A measure-preserving $\Gamma$-action on a standard probability space is called a \textbf{measure-preserving subshift} if there is a full measure Borel $\Gamma$-invariant set $Y$ such that the restriction of the action to $Y$ is a subshift. Note that Bernoulli shifts over finite spaces are examples of measure-preserving subshifts.

Finally, an action by homeomorphisms on a compact Hausdorff space is called a \textbf{topological subshift} if it has a \emph{continuous} generating observable (such an observable must then take values in a finite set by compactness and continuity).

We have the following well-known characterization of subshifts up to conjugacy (in the respective Borel  and topological categories). 

\begin{prop}\label{prop: chara Borel and topo subshifts}
The following hold. 
\begin{enumerate}
\item A Borel action $\Gamma\act X$ is a subshift if and only if it is conjugate to a restriction of $\Gamma\act I^\Gamma$ to an invariant Borel subset where $I$ is a finite set.
\item A $\Gamma$-action by homeomorphisms on a compact zero-dimensional topological space $K$ is a topological subshift if and only if it is conjugate (via a homeomorphism) to the restriction of  $\Gamma\act I^\Gamma$  to an invariant closed subset, where $I$ is a finite set. 
\end{enumerate}
\end{prop}
\begin{proof}
For (1), let $\varphi: X\to I$ be a generating observable where $I$ is finite, and consider the $\Gamma$-equivariant map $\Phi_I:X\to I^\Gamma$ given by $x\mapsto(\varphi_I(\gamma\inv x))_{\gamma\in \Gamma}$. Since $\varphi$ is generating and Borel, $\Phi_I$ is an injective Borel map. Its range is thus a Borel $\Gamma$-invariant subset of $I^\Gamma$. 

For (2), consider the same map and observe that it now has to be a homeomorphism onto its image.
\end{proof}

If we are now given a measure-preserving $\Gamma$-action on $(X,\mu)$, we say that an observable $\varphi$ is \textbf{ generating} if we can find a Borel $\Gamma$-invariant subset $X'\subseteq X$ of full measure such that the restriction of $\varphi$ to $X'$ is generating for $\Gamma\act X'$. The following analogue of Proposition \ref{prop: chara Borel and topo subshifts} is immediate.

\begin{prop}
	A measure-preserving $\Gamma$-action on $(X,\mu)$ is a measure-preserving subshift if and only if it admits a generating observable whose range is finite. \qed
\end{prop}

Let $(X,\mu)$ be standard probability space. 
The \textbf{entropy} of an observable $\varphi: X\to I$ is the non-negative real number $H(\varphi)$ defined by the equation 
\begin{align*}
H(\varphi)&=-\sum_{i\in I}\mu \left(\varphi\inv (\left\{i\right\})\right) \ln\mu \left(\varphi\inv(\left\{i\right\})\right),
\end{align*}
where we make the convention that $0\ln(0)=0$.
Here are two fundamental properties of entropy:
\begin{itemize}
\item Subadditivity: for any two observables $\varphi$ and $\psi$ we have 
\[H(\varphi\times\psi)\leq \H(\varphi)+H(\psi)\]
 where $(\varphi,\psi)$ is the observable $x\mapsto (\varphi(x),\psi(x))$. 
 \item Concavity: if we fix $k\in\N$, the observables taking value in $\{1,...,k\}$ which maximize entropy are exactly those which give the same probability to every $i\in\{1,...,k\}$ (and their entropy is thus $\ln k$)
\end{itemize}

\begin{defi}[Seward]
	The \textbf{Rokhlin entropy} of a measure-preserving aperiodic $\Gamma$-action is the infimum of the entropies of its generating observables. 
\end{defi}

Every aperiodic Borel action admits a generating observable by \cite[Theorem 5.4]{jacksonCountableBorelEquivalence2002}, so the above quantity is an element of $[0,+\infty]$. The following remarkable result of Seward is a far-reaching generalization of Krieger's generator theorem which will allow us to work with arbitrary countable groups rather than $\Z$. 

\begin{theo}[{Seward \cite{sewardErgodicActionsCountable2015}}]\label{thm: 
Seward Krieger}
Let $\Gamma$ be a countable group, let $\Gamma\act(X,\mu)$ be a measure-preserving action. The action is a measure-preserving subshift if and only if it has finite Rokhlin entropy.
\end{theo}

Seward also has a quantitative version of the above result, yielding a 
generating observable whose range is as small as possible (see 
\cite{sewardKriegerfinitegenerator2019}, where he also defines Rokhlin entropy).

\subsection{Full groups}

If $(X,\mu)$ is a standard probability space, and $A,B$ are Borel subsets of $X$, a \textbf{partial isomorphism} of $(X,\mu)$ of \textbf{domain} $A$ and \textbf{range} $B$ is a Borel bijection $f: A\to B$ which is measure-preserving for the measures induced by $\mu$ on $A$ and $B$ respectively. We denote by $\dom f=A$ its domain, and by $\rng f=B$ its range. Given two partial isomorphisms $\varphi_1: A\to B$ and $\varphi_2: C\to D$, we define their composition $\varphi_2\circ \varphi_1$ as the map $\varphi_1\inv(B\cap C)\to \varphi_2(B\cap C)$ given by $\varphi_2\circ\varphi_1(x)=\varphi_1\varphi_2(x)$. We also define the \textbf{inverse} $\varphi\inv$ of a partial isomorphism $\varphi: A\to B$ as the unique map $\varphi\inv: B\to A$ such that $\varphi\inv\circ \varphi=\mathrm{id}_A$.

By definition a \textbf{graphing} is a countable set of partial isomorphisms. Every graphing $\Phi$ \textbf{generates} a \textbf{measure-preserving equivalence relation} $\mathcal R_\Phi$, defined to be the smallest equivalence relation containing $(x,\varphi(x))$ for every $\varphi\in \Phi$ and $x\in\dom\varphi$. Given a graphing $\Phi$, a $\Phi$\textbf{-word} is a composition of finitely many elements of $\Phi$ or their inverses. Observe that $(x,y)\in\mathcal R_\Phi$ iff and only if there exists a $\Phi$-word $w$ such that $y=w(x)$. Moreover we have a natural path metric $d_\Phi$ on the $\mathcal R_\Phi$-classes: for $(x,y)\in\mathcal R_\Phi$ we let $d_\Phi(x,y)$ be the minimum length of a $\Phi$-word $w$ such that $y=w(x)$.

A graphing is called \textbf{aperiodic} when the $\mathcal R_\Phi$-class of $\mu$-almost every $x\in X$ is  infinite.
Given a Borel set $A$, we say that $A$ is $\mathcal R_\Phi$-invariant if whenever $x\in A$ and $(x,y)\in\mathcal R_\Phi$, we have $y\in A$. Equivalently, this means that for all $x\in A$ and all $\varphi\in\Phi$, we have $\varphi(x)\in A$ and $\varphi\inv(x)\in A$. A graphing $\Phi$ is \textbf{ergodic} if every $\mathcal R_\Phi$-invariant Borel set has measure $0$ or $1$. Every ergodic graphing is automatically aperiodic.

The \textbf{full group} of a measure-preserving equivalence relation $\mathcal 
R$ is the group $[\mathcal R]$ of automorphisms of $(X,\mu)$ which induce a 
permutation on every $\mathcal R$-class, that is
$$[\mathcal R]=\{\varphi\in\Aut(X,\mu): \forall x\in X, \varphi(x)\,\mathcal R\, x\}.$$
It is a separable group when equipped with the complete metric $d_u$ defined by $$d_u(T,U)=\mu(\{x\in X: T(x)\neq U(x)\}.$$ The metric $d_u$ is called the \textbf{uniform metric} and the topology it induces is called the uniform topology. One also defines the \textbf{pseudo full group} of $\mathcal R$, denoted by $[[\mathcal R]]$, which consists of all partial isomorphisms $\varphi$ such that $\varphi(x)\, \mathcal R \, x$ for all $x\in \dom\varphi$. 

\subsection{\texorpdfstring{$\LL^1$}{L1} full groups}

\begin{defi}[{\cite{lemaitremeasurableanaloguesmall2018}}]
Let $\Phi$ be a graphing on a standard probability space $(X,\mu)$. Its \textbf{$\LL^1$ full group} $[\Phi]_1$ is the group of all $T\in[\mathcal R_\Phi]$ such that 
\[ \int_Xd_\Phi(x,T(x))d\mu(x)<+\infty. \]
\end{defi}

The $\LL^1$ full group of $\Phi$ is a Polish group for the topology induced by the right-invariant complete metric $d^1_\Phi$ defined by 
$$d^1_\Phi(T,U)=\int_X d_\Phi(T(x),U(x))d\mu(x).$$
When $\Gamma$ is a finitely generated group acting on a standard probability space $(X,\mu)$ by measure-preserving transformations, we can also define the associated $\LL^1$ full group as follows: let $S$ be a finite generating set of the group $\Gamma$, then $S$ is a graphing whose associated equivalence relation is the equivalence relation $\mathcal R_\Gamma$ whose equivalence classes are the $\Gamma$-orbits. 
Observe that up to bilipschitz equivalence, the associated metric $d_S$ on the $\Gamma$-orbits does not depend on $S$, so that two finite generating sets of $\Gamma$ yield the same $\LL¹$ full group with the same topology. This is by definition the $\LL^1$ full group of the $\Gamma$-action, which we  will often simply denote by $[\Gamma]_1$ without referring explicitly to the action. 

Here is another way to think of the $\LL^1$ full group of a $\Gamma$-action, similar to what was done in \cite[Section 3.2]{CarderiMorePolishfull2016}. Consider the Polish space $\LL^1(X,\mu,\Gamma)$ of integrable maps $X\to\Gamma$, i.e. maps $f:X\to\Gamma$ such that $\int_X d_S(1,f(x))d\mu(x)<+\infty$. To each element $f$ of $\LL^1(X,\mu,\Gamma)$ we associate a map $T_f: X\to X$ defined by $T_f(x)=f(x)\cdot x$. The subspace
$$\widetilde{[\Gamma]_1}=\{f\in\LL^1(X,\mu,\Gamma): T_f\in [\mathcal R_\Gamma]\}$$
is closed, and we define on it a law $*$ which is compatible with the group law on $[\mathcal R_\Gamma]$ by letting
$f*g(x)=f(T_g(x))g(x)$. Then $(\widetilde{[\Gamma]_1},*)$ is a Polish group, and $[\Gamma]_1$ is topologically isomorphic to its quotient by the closed normal subgroup $\{f\in \widetilde{[\Gamma]_1}: T_f=\mathrm{id}_X\}$ via the map $f\mapsto T_f$. 

If we are given $T\in [\Gamma]_1$, a lift of $T$ to $\widetilde{[\Gamma]_1}$ is called a \textbf{cocycle} associated to $T$. When the $\Gamma$ action is free, $f\mapsto T_f$ is injective and so each $T\in[\Gamma]_1$ has a unique associated cocycle $c_T: X\to \Gamma$ given by the equation 
$$T(x)=c_T(x)\cdot x.$$

An important feature of cocycles associated to elements of $\LL^1$ full groups is that they have finite entropy. 

\begin{prop}[{Austin, see \cite[Lemma 2.1]{austinBehaviourEntropyBounded2016}}]\label{prop: L1 has finite entropy}
Let $\Gamma$ be a finitely generated group. Then every element of $\LL^1(X,\mu,\Gamma)$ has finite entropy. 
\end{prop}

We finally note the following lemma about cocycles which are measurable with 
respect to an invariant sub-$\sigma$-algebra.

\begin{lemm}\label{lem: measurable cocycle}
	Let $\Gamma\act(X,\mu)$ be a measure-preserving action of a finitely 
	generated 
	group $\Gamma$, suppose $\mathcal A$ is a $\Gamma$-invariant sub-
	$\sigma$-algebra of the Borel subsets of $X$. Let $[\Gamma]_1^{\mathcal A}$ 
	denote the set consisting of all elements of 
	$[\Gamma]_1$ which admit an $\mathcal A$-measurable cocycle. Then 
	$[\Gamma]_1^{\mathcal A}$ is a group.
\end{lemm}
\begin{proof}
	Let $T,U\in [\Gamma]_1^{\mathcal A}$, and
	let $c_T$ and $c_U$ be cocycles for $T$ and $U$ which are $\mathcal 
	A$-measurable. Then the map $c:X\to \Gamma$ defined by 
	$$c(x)=c_T(U(x))c_U(x)$$
	is a cocycle for $TU$, and $c(x)=\gamma$ if and only if there are 
	$\gamma_1,\gamma_2\in\Gamma$ such that $\gamma_1\gamma_2=\gamma$, 
	$c_U(x)=\gamma_2$, $c_T(\gamma_2 x)=\gamma_1$. So $c$ is clearly $\mathcal 
	A$-measurable as wanted.
\end{proof}

\begin{rema}
	The same proof yields a similar statement for the whole full group of 
	the action.
\end{rema}

\subsection{Derived \texorpdfstring{$\LL^1$}{L1} full groups}

Contrarily to many Polish groups that have been studied, $\LL^1$ full groups are not topologically simple in general, and can actually happen to have a non trivial topological abelianization. The closure of their derived group (i.e. the group generated by commutators) is thus distinct from them a priori. 
Let us give it a shorter name.

\begin{defi}The \textbf{derived} $\LL^1$ \textbf{full group} of a graphing 
$\Phi$ is the closure of the derived group of the $\LL^1$ full group of $\Phi$. 
It is denoted by $[\Phi]_1'$. 
\end{defi}

Given $\varphi\in\Phi$ and a Borel set $A\subseteq X$, let us define an 
involution 
$I_{\varphi,A}$ by: 
$$I_{\varphi,A}(x)=\left\{\begin{array}{cc}\varphi(x) & \text{ if }x\in 
A\setminus\varphi(A) \\\varphi\inv(x) & \text{ if }x\in\varphi(A)\setminus A 
\\x & \text{ else.}\end{array}\right.$$
We can now give two fundamental properties of the derived $\LL^1$ full group. 
The second one suggests to rather call it the \emph{symmetric $\LL^1$ full 
group} (see Section \ref{sec: symmetric lp full group}).

\begin{theo}[{\cite[Lemma 3.11 and Theorem 
3.15]{lemaitremeasurableanaloguesmall2018}}]\label{thm: topo gen by I_phi}
	Let $\Phi$ be an aperiodic graphing. Then the following hold:\begin{enumerate}[(a)]
		\item every periodic element of the $\LL^1$ full group of $\Phi$ actually belongs to the derived $\LL^1$ full group of $\Phi$;
		\item the derived $\LL^1$ full group of $\Phi$ is topologically 
		generated by the set of involutions $\{I_{\varphi,A}:\varphi\in\Phi, 
		A\subseteq X\}$.
	\end{enumerate}
\end{theo}
Observe that since $I_{\varphi,A}$ depends continuously on $A\in\MAlg(X,\mu)$ 
and $\MAlg(X,\mu)$ is connected, we can deduce that $[\Phi]_1'$ is connected 
(cf. Corollary 3.14 in \cite{lemaitremeasurableanaloguesmall2018}). We still 
don't know whether it is simply connected. Another important question is to 
understand what is the topological abelianization of $\LL^1$ full groups, i.e. 
what is $[\Phi]_1/[\Phi]_1'$. The only known case is that of ergodic 
$\Z$-actions, for which we then always have $[\Z]_1/[\Z]_1'=\Z$ 
\cite{lemaitremeasurableanaloguesmall2018}. See the last section of the paper 
for more on this. 

Using Item (b) from the above statement along with its full group version 
\cite[Section 4]{kittrellTopologicalPropertiesFull2010}, we also 
have the following.

\begin{coro}\label{cor: density}
	Let $\Phi$ be an aperiodic graphing. Then the derived $\LL^1$ full group of 
	$\Phi$ is dense in the full group $[\mathcal R_\Phi]$ for the uniform 
	topology.
\end{coro}

Given an aperiodic graphing $\Phi$ and an $\mathcal R_\Phi$-invariant Borel set $A$, consider the subgroup $[\Phi_A]'_1$ consisting of all the elements of $[\Phi]'_1$ supported in $A$. It is a closed normal subgroup of $[\Phi]'_1$, and the following result says that these are the only ones.

\begin{theo}[{\cite[Theorem 3.24]{lemaitremeasurableanaloguesmall2018}}]
Let $\Phi$ be an aperiodic graphing. Then every closed normal subgroup of $[\Phi]_1'$ is of the form $[\Phi_A]'_1$, where $A$ is $\mathcal R_\Phi$-invariant. 
\end{theo}

As a corollary, one can show that a graphing is ergodic if and only if its 
derived $\LL^1$ full group is topologically simple \cite[Corollary 
3.25]{lemaitremeasurableanaloguesmall2018}.

\subsection{Conditional measures}

Given a measure preserving action $\Gamma\act(X,\mu)$ of a finitely generated group $\Gamma$, it will be important for us to know for which Borel sets $A$ and $B$ we can find an element of $T$ of the $\LL^1$ full group of the action such that $T(A)$ is close to $B$. Because the $\LL^1$ full group is dense in the full group $[\mathcal R_\Gamma]$, it suffices to understand the $[\mathcal R_\Gamma]$-orbits in the measure algebra. It is a well-known fact that when the action is ergodic, there is $T\in[\mathcal R_\Gamma]$ such that $T(A)=B$ if and only if $\mu(A)=\mu(B)$. In the non ergodic case, one needs to refine the measure $\mu$ to a \emph{conditional measure} with respect to the $\sigma$-algebra of $\Gamma$-invariant sets. 

\begin{defi}
	Let $\mathcal R$ be a measure-preserving equivalence relation. Denote by $M_{\mathcal R}$ the $\sigma$-algebra of $\mathcal R$-invariant Borel subsets. The \textbf{conditional measure} of $A\in\MAlg(X,\mu)$ is the $M_{\mathcal R}$-measurable function $\mu_{\mathcal R}$ defined by 
	$$\mu_{\mathcal R}(A)=\pi_{\mathcal R}(\chi_A),$$
	where $\pi_{\mathcal R}$ is orthogonal projection onto the closed subspace of $\LL^2(X,\mu)$ consisting of the $M_{\mathcal R}$-measurable functions. 
\end{defi}

Note that $\mu_{\mathcal R}$ is well-defined up to measure zero. We now have the following. 

\begin{prop}[{Dye, see \cite[Lemma 2.11]{lemaitremeasurableanaloguesmall2018}}]\label{prop: transitive on cond measure}Let $\mathcal R$ be a measure-preserving equivalence relation and $A,B\in\MAlg(X,\mu)$. Then $\mu_{\mathcal R}(A)=\mu_{\mathcal R}(B)$ if and only if there is an involution $T\in[\mathcal R]$ such that $T(A)=B$.
\end{prop}

Aperiodicity has an important reformulation in terms of conditional measures, which is also due to Dye who calls it Maharam's lemma. For a proof in our setup, see \cite[Proposition 2.2]{lemaitrefullgroupsnonergodic2016}.

\begin{prop}[{Maharam's lemma}]\label{prop: maha}
	A measure-preserving equivalence relation $\mathcal R$ is aperiodic if and only if for any $A\in\MAlg(X,\mu)$, and for any $M_\mathcal R$-measurable function $f$ such that $0\leq f\leq \mathbb \mu_{\mathcal R}(A)$, there exists $B\subseteq A$ such that the $\mathcal R$-conditional measure of $B$ equals $f$.
\end{prop}

\subsection{Alternating topological full groups}\label{sec: topo full def}

Alternating topological full groups were defined by Nekrashevych and have been the source of outstanding new examples of finitely generated groups \cite{NekrashevychPalindromicsubshiftssimple2018}. They provide the right setup for generalizing Matui's results on finite generatedness of derived groups of topological full groups of homeomorphisms of the Cantor space \cite{Matuiremarkstopologicalfull2006}. Topological full groups were first studied and defined by Giordano, Putnam and Skau \cite{GiordanoFullgroupsCantor1999}. 
Here we give a restricted definition of (alternating) topological full groups which will be sufficient for our purposes.  

Let $2^\N$ be the Cantor space and let $\nu$ be a Borel probability measure on it. Recall from \cite{lemaitremeasurableanaloguesmall2018} that a \textbf{continuous graphing} on the measured Cantor space $(2^\N,\nu)$ is a graphing $\Phi$ such that for all $\varphi\in\Phi$, the sets $\dom \varphi$, $\rng\varphi$ are clopen and $\varphi: \dom\varphi\to \rng\varphi$ is a homeomorphism.

\begin{defi}Let $\Phi$ be a continuous graphing on $(2^\N,\nu)$. The 
\textbf{topological full group} of $\Phi$, denoted by $[\Phi]_c$, is the group 
of all homeomorphisms $T$ of $2^\N$ such that for all $x_0\in 2^\N$ there is a 
neighborhood $U$ of $x_0$ and a $\Phi$-word $w$ such that for all $x\in U$, 
$T(x)=w(x)$. 
\end{defi}

Say that an element $T$ of the topological full group of $\Phi$ is a \textbf{basic $3$-cycle} if there are a three clopen subsets $U_1, U_2, U_3$ which partition the support of $T$ and two $\Phi$-words $w_1,w_2$ such that $w_1(U_1)=U_2$ and $w_2(U_2)=U_3$ and for all $x\in X$ we have 
\[
T(x)=\left\{ 
\begin{array}{cl}
w_1(x)&\text{ if }x\in U_1\\
w_2(x)&\text{ if }x\in U_2\\
w_1\inv w_2\inv(x)&\text{ if }x\in U_3\\
x&\text{ otherwise.}
\end{array}
\right.
\]
The \textbf{alternating topological full group} of a continuous graphing $\Phi$ is the group $\mathfrak A[\Phi]$ generated by basic $3$-cycles. One can check that the alternating topological full group is a normal subgroup contained in the derived group of the topological full group. 

Let us observe that when $\Gamma$ is a countable group acting by homeomorphisms on the Cantor space, then if $S$ is any generating set for $\Gamma$, the elements of $S$ define a continuous graphing whose topological full group is the topological full group of the groupoid of germs of the $\Gamma$-action in the sense of \cite[Example 2.2]{nekrashevychSimplegroupsdynamical2017} (see also \cite[Definition 2.6]{nekrashevychSimplegroupsdynamical2017} and the two paragraphs which follow it). Moreover, the alternating topological full group of the continuous graphing $S$ is equal to the alternating topological full group of the groupoid of germs of the $\Gamma$-action.

\begin{rema}
	Since we do not require the action to be free or essentially free in the topological sense, the etale groupoid of germs of the action is not Hausdorff in general, in particular its unit space may not be closed, although it is compact. As a consequence, Lemma 2.2 from \cite{matuiTopologicalFullGroups2013} does not apply and elements of the topological full group may very well have a non clopen support\footnote{In the topological setup, the \textbf{support} of a homeomorphism $T:X\to X$ is rather defined as the \emph{closure} as the set of $x\in X$ such that $T(x)\neq x$.}, as opposed to elements of the alternating topological full group.
\end{rema}

We have the following straightforward consequence of \cite[Theorem 1.2]{nekrashevychSimplegroupsdynamical2017} and \cite[Proposition 5.7]{nekrashevychSimplegroupsdynamical2017} which will be crucial in our proof of Theorem \ref{thmi: finite topo rank}.

\begin{theo}[{Nekrashevych}]\label{thm: Nekrashevych expansive vs finite rank}
	Let $\Gamma$ be a finitely generated group, suppose $\Gamma$ acts on the Cantor space by homeomorphism and that every $\Gamma$-orbit has cardinality at least $5$. Then the action is a topological subshift if and only if its alternating topological full group is finitely generated.
\end{theo}

\section{Density of the alternating topological full group}\label{sec: dense topo full}

In this section, we show that the alternating topological full group of a 
continuous graphing is always dense in its derived $\LL^1$ full group. We first 
need to work a bit on the general setup, so $(X,\mu)$ is still a standard 
probability space.

A measure-preserving transformation is called a \textbf{$3$-cycle} when it has order $3$, or equivalently when all its nontrivial orbits have cardinality $3$. Given a $3$-cycle $T$, there is $A\subseteq X$ such that $\supp T=A\sqcup T(A)\sqcup T^2(A)$, and the identity $(1\quad 2\quad 3)=(1\quad 2)(1\quad 3)(1\quad 2)(1\quad 3)$ yields that $T$ is a commutator, so the $3$-cycles of the $\LL^1$ full group actually belong to the derived $\LL^1$ full group.  

We need to ensure that there are many $3$-cycles. This is a consequence of the 
following two lemmas, which are straightforward generalizations of \cite[Lemma 
3.6]{lemaitremeasurableanaloguesmall2018} and \cite[Lemma 
2.12]{lemaitremeasurableanaloguesmall2018} respectively. 

Recall the following notation: given a measure-preserving transformation $T$ 
and $A\subseteq X$, we denote by $T_A$ the transformation induced by $T$ on $A$ 
(see Section \ref{sec: meas alg and auto}). In what follows, we 
only use this construction when $A$ is $T$-invariant, and then $T_A$ is the map 
which acts as $T$ on $A$ and trivially outside of $A$.

\begin{lemm}\label{lem: manyperiodic}
Let $\Phi$ be a graphing. Then for every periodic $U\in[\mathcal R_\Phi]$, there exists an increasing sequence of $U$-invariant sets $A_n\subseteq \supp U$ such that $\supp U=\bigcup_{n\in\N} A_n$ and for all $n\in\N$, $U_{A_n}\in[\Phi]_1$.
\end{lemm}
\begin{proof}
Let $U\in[\mathcal R_\Phi]$ be periodic and for all $n\in\N$, let $A_n=\{x\in \supp U: d_\Phi(x,U^k(x))<n\text{ for all }k>0\}$. By definition each $A_n$ is $U$-invariant, and since $U$ is periodic, $\bigcup_{n\in\N} A_n=\supp U$, so $\mu(\supp U\setminus A_n)\to 0$. By the definition of $A_n$, each $U_{A_n}$ belongs to $[\Phi]_1$.
\end{proof}

\begin{lemm}\label{lem: many3cycles}
Let $\Phi$ be an aperiodic graphing. Then for every Borel $A\subseteq X$, there 
is a $3$-cycle $U\in [\mathcal R_\Phi]$ whose support is equal to $A$. 
\end{lemm}
\begin{proof}
By Maharam's Lemma (Proposition \ref{prop: maha}), one can write $A=A_1\sqcup A_2\sqcup A_3$ where $\mu_{\mathcal R_\Phi}(A_1)=\mu_{\mathcal R_\Phi}(A_2)=\mu_{\mathcal R_\Phi}(A_3)=\mu_{\mathcal R_\Phi}(A)/3$. By a result of Dye (see Proposition \ref{prop: transitive on cond measure}), there are $T_1,T_2\in [\mathcal R_\Phi]$ such that $T_1(A_1)=A_2$ and $T_2(A_2)=A_3$). We then define $T\in[\mathcal R_\Phi]$ by
\[
T(x)=
\left\{\begin{array}{cl}
T_1(x) & \text{if }x\in A_1 \\
T_2(x) & \text{if }x\in A_2 \\
T_1\inv T_2\inv(x) & \text{if }x\in A_3 \\
x & \text{otherwise.}
\end{array}\right.
\]
The transformation $T$ is the desired $3$-cycle. 
\end{proof}
As an immediate consequence of the two previous lemmas, we have:

\begin{prop}\label{prop: manythreecycles}
Let $\Phi$ be an aperiodic graphing. Then for every Borel $A\subseteq X$, there 
is a sequence of $3$-cycles $U_n\in [\Phi]_1$ such that $(\supp U_n)$ is an 
increasing sequence of subsets of $A$ and $\mu(A\setminus \supp U_n)\to 0$. \qed
\end{prop}

\begin{prop}\label{prop: derived generated by 3-cycles}	
Let $\Phi$ be a graphing. Then $[\Phi]_1'$ is topologically generated by $3$-cycles. 
\end{prop}
\begin{proof}	
Let $X_p$ be the periodic part of the graphing, i.e. the set of $x$ whose $\mathcal R_\Phi$-class is finite, and let $X_a=X\setminus X_p$. Then if one lets $\Phi_p$
 (resp. $\Phi_a$) be the restriction of $\Phi$ to $X_p$ (resp. $X_a$), the derived $\LL^1$ full group of $\Phi$ splits naturally as a direct product $[\Phi_p]_1'\times [\Phi_a]_1'$. So it suffices to establish the theorem in the aperiodic and in the periodic case. 
 
Let us start by supposing $\Phi$ is aperiodic. The group generated by $3$-cycles is normal in $[\Phi]_1'$ so its closure $N$ is a closed normal subgroup of $[\Phi]_1'$. By \cite{lemaitremeasurableanaloguesmall2018} we conclude that $N$ is equal to the group of $T\in[\Phi]'_1$ whose support is contained in the reunion of the supports of the $3$-cycles. Using the previous proposition, we see that the reunion of such supports is equal to $X$, so $N=[\Phi]'_1$, which end the proof of the theorem in the aperiodic case.

Now suppose $\Phi$ is periodic. For every $n\in\N$ let $X_n$ be the set of $x\in X$ whose $\mathcal R_\Phi$-class has $d_\Phi$-diameter less than $n$. Then $X=\bigcup_{n\in\N} X_n$, the $X_n$'s are $\Phi$-invariant and non-decreasing so if we let $G_n=\{T\in [\Phi]_1: \supp T\subseteq X_n\}$ then the reunion of the $G_n$'s is dense in $[\Phi]_1$. It thus suffices to show that commutator of elements of $G_n$ is a product of $3$-cycles belonging to $G_n$, which comes from the fact that by boundedness of the orbits, each $G_n$ is a full group and every periodic element of a full group which is a commutator is a product of $3$-cycles. 
\end{proof}

\begin{prop}\label{prop: alternating dense in derived L1}
Let $\Phi$ be a continuous graphing on $(2^\N,\nu)$. Then its alternating topological full group is dense in its derived $\LL^1$ full group. 
\end{prop}
\begin{proof}
By the previous proposition, it suffices to show that every $3$-cycle in $[\Phi]_1$ can be approximated by an element of the alternating topological full group. So let $T$ be a $3$-cycle in the $\LL^1$ full group of $\Phi$. Let $A\subseteq X$ such that $\supp T=A\sqcup T(A)\sqcup T^2(A)$. 

Since the set of $\Phi$-words is countable, we have a partition $(A_n)$ of $A$ and $\Phi$-words $w_{1,n}, w_{2,n}$ such that for all $x\in X$ we have
\[
T(x)=
\left\{\begin{array}{cl}
w_{1,n}(x) & \text{if }x\in A_n \text{ for some }n\in\N \\
w_{2,n}(x) & \text{if }x\in T(A_n) \text{ for some }n\in\N \\
w_{1,n}\inv w_{2,n}\inv(x) & \text{if }x\in T^2(A_n) \text{ for some }n\in\N\\
x & \text{otherwise.}
\end{array}\right.
\]
In particular, for every $n\in\N$ the set $A_n\sqcup T(A_n)\sqcup T^2(A_n)$ is $T$-invariant, and if we let $T_n:=T_{A_n\sqcup T(A_n)\sqcup T^2(A_n)}$, then by the dominated convergence theorem we conclude $T=\lim_{m\to+\infty}\prod_{i=0}^m T_m$. So it suffices to prove that each $T_n$ is a limit of elements of the alternating topological full group. 

Let $n\in\N$, then there is a sequence of clopen sets $(U_i)_{i\in\N}$ such that $\mu(A_n\bigtriangleup U_i)\to 0\;[i\to+\infty]$. If we let $V_i=U_i\setminus (T(U_i)\cup T^2(U_i))$, then by continuity we also have $\mu(A_n\bigtriangleup V_i)\to 0$. We now define $T_{n,i}$ by 
\[
T_{n,i}(x)=
\left\{\begin{array}{cl}
w_{1,n}(x) & \text{if }x\in V_i \\
w_{2,n}(x) & \text{if }x\in T(V_i) \\
w_{1,n}\inv w_{2,n}\inv(x) & \text{if }x\in T^2(V_i) \\
x & \text{otherwise.}
\end{array}\right.
\]
Then $T_{n,i}$ is clearly a basic $3$-cycle, and hence belongs to the alternating topological full group of $\Phi$. Moreover by the dominated convergence theorem we have $T_{n,i}\to T_n\;[n\to+\infty]$, which finishes the proof. 
\end{proof}

\section{Topological rank of derived \texorpdfstring{$\LL^1$}{L1} full groups}\label{sec: proof of finite topo rank}

In this section, we prove that the derived $\LL^1$ full group of a measure-preserving $\Gamma$-action has finite topological rank if and only if it is a subshift (Theorem \ref{thmi: finite topo rank}). We want to use Nekrashevych's result on finite generatedness of alternating topological full groups, and the next section prepares the ground for this.

\subsection{From measure-preserving subshifts to topological subshifts}

In this preparatory section we will show that every aperiodic measure-preserving subshift is conjugate to a topological subshift on the Cantor space where each orbit has cardinality at least $5$. For this, we will freely use the characterization of subshifts provided by Proposition \ref{prop: chara Borel and topo subshifts}.

Given a group $\Gamma$ and a generating set $S$ for $\Gamma$, we equip $\Gamma$ with the graph metric induced by its Cayley graph with respect to $S$. The closed ball of radius $n$ around the identity for this metric is denoted by $B_S(e,n)$. 

\begin{lemm}
Let $\Gamma$ be a finitely generated group, suppose $S\subseteq\Gamma$ is a finite generating set and let $\Gamma\act X$ be a subshift with every orbit of cardinality at least $5$. Then we can find a finite generating Borel partition $\mathcal P$ such that for every $P\in\mathcal P$ there are $\gamma_1,...,\gamma_4\in B_S(e,4)$ such that $P, \gamma_1 P,...,\gamma_4 P$ are disjoint. 
\end{lemm}
\begin{proof}
Let us first show that for every $x\in X$ there are $\gamma_1,...,\gamma_4\in B_S(e,4)$ such that $x,\gamma_1 x,...,\gamma_4x$ are all distinct. Observe that if for some $n\in \N$ we have $\abs{B_S(e,n) x}=\abs{B_S(e,n+1)x}$ then $B_S(e,n)x$ is a finite set which is invariant under each generator of $\Gamma$, hence $\Gamma$-invariant. In particular we must have  $\abs{B_S(e,4) x}\geq 5$ so for every $x\in X$ there are $\gamma_1,...,\gamma_4\in B_S(e,4)$ such that $x,\gamma_1 x,...,\gamma_4x$ are all distinct.

Now apply lemma \ref{lem: partitioning intersection of supports} to the family $(\gamma\lambda\inv)_{\gamma,\lambda\in B_S(e,4)}$ to get a partition $\mathcal P$ such that for every $P\in\mathcal P$ there are $\gamma_1,...,\gamma_4\in B_S(e,4)$ such that $P, \gamma_1 P,...,\gamma_4 P$ are disjoint. The desired partition is obtained by taking the join of $\mathcal P$ with a finite generating partition. 
\end{proof}

\begin{prop}\label{prop: topo subshift with not too small orbits}
Let $\Gamma$ be a finitely generated group acting on $(X,\mu)$ by an aperiodic measure-preserving subshift. Then the action is conjugate to a topological subshift on the Cantor space where every orbit has cardinality at least $5$.
\end{prop}
\begin{proof}
Let $\mathcal P$ be as in the previous lemma and consider $\Omega=\bigvee_{\gamma\in B_S(e,4)}\gamma\mathcal P$. Observe that for $\mu$-almost every $x\in X$, there are $\gamma_1,...,\gamma_4\in B_S(e,4)$ such that $x$, $\gamma_1 x$,...,$\gamma_4 x$ all belong to distinct elements of $\Omega$. Let $\varphi_\Omega:X\to\Omega$ be the associated observable. 

Now consider the $\Gamma$-equivariant map $\Phi_\Omega:X\to \Omega^\Gamma$ given by $x\mapsto(\varphi_\Omega(\gamma\inv x))_{\gamma\in \Gamma}$. Observe that the set of sequences $(\omega_\gamma)_{\gamma\in\Gamma}$ such that for all $\gamma\in \Gamma$ there are $\gamma_1,...,\gamma_4\in B_S(e,4)$ such that $\omega_\gamma,\omega_{\gamma\gamma_1\inv},...,\omega_{\gamma\gamma_4\inv}$ are all distinct is a closed $\Gamma$-invariant set containing the essential range of $\Phi_\Omega$, so the $\Gamma$-action on the essential range of $\Phi_\Omega$ is a topological subshift all whose orbits have cardinality at least $5$. 

Moreover the measure $\mu$ has no atoms because the $\Gamma$-orbits are almost all infinite. So the essential range of $\Phi_\Omega$ contains no isolated points and thus must be homeomorphic to the Cantor space by a theorem of Brouwer (see \cite[Theorem 7.4]{kechrisClassicaldescriptiveset1995}).
We conclude that the $\Gamma$-action on the essential range of $\Phi_\Omega$ is a topological subshift on the Cantor space which has the desired properties. 
\end{proof}

\begin{rema}
	I do not know whether one can improve the above result so as to get an 
	aperiodic topological subshift.
\end{rema}
\subsection{Proof of Theorem \ref{thmi: finite topo rank}}

Let us recall the statement that we are aiming at.

\begin{theo}Let $\Gamma$ be a finitely generated group and consider a 
measure-preserving aperiodic $\Gamma$-action on a standard probability space 
$(X,\mu)$. Then the following are equivalent:
\begin{enumerate}[(i)]
\item \label{cond: finite entropy}the $\Gamma$-action has finite Rokhlin entropy and
\item \label{cond: finite topological rank}the derived $\LL^1$ full group of the $\Gamma$-action has finite topological rank.
\end{enumerate}
\end{theo}
\begin{proof}
Let us start by the direct implication \eqref{cond: finite entropy}$\impl$\eqref{cond: finite topological rank}. Suppose the measure-preserving $\Gamma$-action has finite Rokhlin entropy. By Seward's theorem (Theorem \ref{thm: Seward Krieger}), we may assume that the action is a measurable subshift, and by Proposition \ref{prop: topo subshift with not too small orbits} we can further assume that it is a topological subshift all whose orbits have cardinality at least $5$. Nekrashevych's result (Theorem \ref{thm: Nekrashevych expansive vs finite rank}) yields that the associated alternating topological full group is finitely generated. So the alternating topological full group is the desired dense finitely generated subgroup of the derived $\LL^1$ full group of the action by Proposition \ref{prop: alternating dense in derived L1}.\\

We now prove the reverse implication \eqref{cond: finite topological rank}$\impl$\eqref{cond: finite entropy}. Suppose that for some $n\in\N$ we have $T_1$,..., $T_n \in [\Gamma]'_1$ which generate a dense subgroup of $[\Gamma]'_1$. By definition for each $i\in \{1,...,n\}$ the transformation $T_i$ admits an associated cocycle $c_i\in \LL^1(X,\mu,\Gamma)$. By Lemma \ref{prop: L1 has finite entropy} each cocycle $c_i$ has finite entropy. Finally, let $c$ be the observable associated to a partition $(A_1,A_2)$ of $X$ into two disjoint sets such that for all $x\in X$, $\mu_{\mathcal R_\Gamma}(A_1)(x)=\mu_{\mathcal R_\Gamma}(A_2)(x)=\frac 12$, as provided by Maharam's lemma (cf. Proposition \ref{prop: maha}).

The observable $c_1\times\cdots\times c_n\times c$ has finite entropy by 
subadditivity, let us show that it is generating so that \eqref{cond: finite 
entropy} holds. Let $\tilde{\mathcal A}$ be the $\sigma$-algebra generated by 
the $\Gamma$-translates of the partition associated to the observable 
$c_1\times\cdots\times c_n\times c$, we will show that it separates the points 
of a full measure subset of $X$, thus finishing the proof that the action has 
finite Rokhlin entropy.

First, by Lemma \ref{lem: measurable cocycle}, every
element of the group generated by $T_1,...,T_n$ admits a cocycle which is 
$\tilde {\mathcal A}$-measurable. Furthermore, the following holds.

\begin{claim}The $\sigma$-algebra $\tilde {\mathcal A}$ is invariant under every element of the group generated by $T_1,...,T_n$.
\end{claim}
\begin{claimproof} 
We just saw that every element of the group generated by $T_1,...,T_n$ admits a 
cocycle which is $\tilde{\mathcal A}$-measurable. 
But such elements must preserve $\tilde {\mathcal A}$ because if $A\in\tilde{\mathcal A}$ and $U$ is an element of the full group of $\mathcal R_\Gamma$ admitting a cocycle which is $\tilde{\mathcal A}$-measurable, we find a partition $(A_\gamma)_{\gamma\in\Gamma}$ of $A$ such that each $A_\gamma$ belongs to $\tilde{\mathcal A}$ and for each $x\in A_\gamma$, $U(x)=\gamma x$.
 So $U(A)=\bigsqcup_{\gamma\in\Gamma} \gamma(A_\gamma)$ and since $\tilde{\mathcal A}$ is a $\Gamma$-invariant $\sigma$-algebra which contains each $A_\gamma$ we conclude that $U(A)\in\tilde{\mathcal A}$. 
\end{claimproof}
Let $\mathcal A$ be the closure of the image of $\tilde{\mathcal A}$ in 
$\MAlg(X,\mu)$\footnote{Actually $\tilde{\mathcal A}$ is automatically closed 
because $\tilde{ \mathcal A}$ is a $\sigma$-algebra, but we won't need that.}, 
which is then also invariant under every element of the group generated by 
$T_1,...,T_n$. By Lemma \ref{lem: separating and density}, in order to show 
that $\tilde{\mathcal A}$ separates the points of a full measure subset of 
$X$,  we only need to show that $\mathcal A=\MAlg(X,\mu)$.
We first prove that $\mathcal A$ satisfies the following property with respect to the $\mathcal R_\Gamma$-conditional measure $\mu_{\mathcal R_\Gamma}$: 
\begin{equation}\label{eq: transitivity of measure}
\forall A\in \mathcal A\ \forall B\in \MAlg(X,\mu),  \mu_{\mathcal R_\Gamma}(A)=\mu_{\mathcal R_\Gamma}(B) \impl B\in\mathcal A.
\end{equation}
Indeed let $A\in \mathcal A$, and suppose $B\in \MAlg(X,\mu)$ satisfies $\mu_{\mathcal R_\Gamma}(A)=\mu_{\mathcal R_\Gamma}(B)$. By Proposition \ref{prop: transitive on cond measure} we may find an involution $U\in[\mathcal R_\Gamma]$ such that $U(A)=B$. Using Lemma \ref{lem: manyperiodic}, we then find a sequence $(U_k)_{k\in\N}$ of elements of the group generated by $T_1,...,T_n$ such that $\mu(U_k(A)\bigtriangleup B)\to 0$. Since $\mathcal A$ is closed and invariant under elements of the group generated by $T_1,...,T_n$, we conclude that $B\in\mathcal A$.

Now observe recall that by the definition of our observable $c$, the algebra $\mathcal A$ contains a set $B$ whose $\mathcal R_\Gamma$-conditional measure is constant equal to $1/2$. 
By Maharam's lemma, whenever $A\in\MAlg(X,\mu)$ satisfies $\mu_{\mathcal 
R_\Gamma}(A)(x)\leq 1/2$ for all $x\in X$, we can write $A$ as the intersection 
of two elements of $\MAlg(X,\mu)$ whose $\mathcal R_\Gamma$-conditional measure 
is constant equal to $1/2$. So by \eqref{eq: transitivity of measure}, 
$\mathcal A$ contains all sets whose $\mathcal R_\Gamma$-conditional measure is 
everywhere at most $1/2$. Using Maharam's lemma once more, we see that every 
element of $\MAlg(X,\mu)$ can be written as the union of two elements of 
$\MAlg(X,\mu)$ whose $\mathcal R_\Gamma$-conditional measure is everywhere at 
most $1/2$, so we conclude $\mathcal A=\MAlg(X,\mu)$ as wanted.
\end{proof}

\section{Amenability and \texorpdfstring{$\LL^1$}{L1} full groups}\label{sec:amenability}

Recall that a Polish group $G$ is \textbf{extremely amenable} when every continuous $G$-action on a compact Polish space has a fixed point, while $G$ is \textbf{amenable} if every continuous $G$-action on a compact Polish space admits an invariant Borel probability measure. Following \cite{PestovConcentrationmeasurewhirly2009} we say that a Polish group $G$ is \textbf{whirly amenable} if it is amenable and, given any continuous $G$-action on a compact Polish space $K$, the support of every invariant Borel probability measure is contained in the set of fixed points. 

One has the following implications: whirly amenability implies extreme 
amenability which implies amenability. The first examples of whirly amenable 
groups were found by Y. Glasner, B. Tsirelson and B. Weiss, who showed that 
\emph{Lévy groups} are  whirly amenable \cite[Thm. 
1.1]{GlasnerautomorphismgroupGaussian2005}. Note that if $G$ is a Polish group 
which contains a dense increasing union of whirly amenable groups, then $G$ is 
whirly amenable.

Let $\Phi$ be an amenable graphing, our aim is to show that its derived $\LL^1$ 
full group is whirly amenable. Recall that a measure-preserving equivalence 
relation is called  \textbf{finite} if all its classes are finite. The two 
following lemmas are essentially contained in 
\cite{dowerkBoundedNormalGeneration2019}, but we include a full proof for the 
reader's convenience.

\begin{lemm}\label{lem: dense increasing union}Let $\Phi$ be an aperiodic 
graphing, suppose $\mathcal R_\Phi$ is written as an increasing union of 
equivalence relations $\mathcal R_n$. Then
$\bigcup_{n\in\N}[\mathcal R_n]\cap [\Phi]_1'$
is dense in $[\Phi]_1'$. 
\end{lemm}
\begin{proof}
By Theorem \ref{thm: topo gen by I_phi} we only need to be able to approximate involutions from  $[\Phi]_1'$ by elements from $\bigcup_{n\in\N}[\mathcal R_n]\cap [\Phi]_1'$. Let $U\in [\Phi]_1'$ be an involution, then the $U$-invariant sets $A_n=\{x\in X: x\mathcal R_n U(x)\}$ satisfy $\bigcup_{n\in\N} A_n= X$ since $\bigcup_n \mathcal R_n=\mathcal R_\Phi$. By the dominated convergence theorem, we then have $U_{A_n}\to U$. Moreover $U_{A_n}$ is an involution so $U_{A_n}\in [\Phi]_1'$ and by construction $U_{A_n}\in[\mathcal R_n]$ so $U_{A_n}\in [\mathcal R_n]\cap [\Phi]_1'$ and the lemma is proved.
\end{proof}

\begin{lemm}\label{lem: bounded classes}Let $\Phi$ be an aperiodic graphing, 
let $\mathcal R$ be a finite subequivalence relation of $\mathcal R_\Phi$ whose 
equivalence classes have a uniformly bounded $\Phi$-diameter. Then $[\mathcal 
R]$ is a closed subgroup of $[\Phi]_1'$. 
\end{lemm}
\begin{proof}
The fact that there is a uniform bound $M$ on the diameter of the $\mathcal R$-classes ensures us that $[\mathcal R]$ is a subgroup of $[\Phi]_1$. Moreover, the inclusion map is $M$-lipschitz (for the uniform distance on $[\mathcal R]$ and $d^1_\Phi$ on $[\Phi]_1$) and its inverse is clearly $1$-lipschitz, so $[\mathcal R]$ is closed in $[\Phi]_1$. Finally, $[\mathcal R]$ only consists of periodic elements, and since these belong to $[\Phi]_1'$ (see \cite[Lemma 3.10]{lemaitremeasurableanaloguesmall2018}), we conclude that $[\mathcal R]$ is a closed subgroup of $[\Phi]_1'$.
\end{proof}

\begin{theo}Let $\Phi$ be an amenable aperiodic graphing. Then $[\Phi]_1'$ is 
whirly amenable and $[\Phi]_1$ is amenable. 
\end{theo}
\begin{proof}
Since $[\Phi]_1$ is an extension of $[\Phi]_1'$ by an abelian (hence amenable) group, the second statement follows from the first, so we only need to prove that $[\Phi]_1'$ is whirly amenable. 

By the Connes-Feldmann-Weiss theorem \cite{connesamenableequivalencerelation1981a}, we may write $\mathcal R_\Phi$ as an increasing union of finite equivalence relations $\mathcal R_n$. Let us now change these so that their classes have uniformly bounded diameter. 

We first find an increasing sequence of integers $(\varphi(n))_{n\in\N}$ such that for all $n\in\N$, $$\mu(\{x\in X: \diam_\Phi([x]_{\mathcal R_n})>\varphi(n)\})<\frac 1{2^n}.$$
Then by the Borel-Cantelli lemma, for almost all $x\in X$ there are only finitely many $n\in\N$ such that $\diam_\Phi([x]_{\mathcal R_n})>\varphi(n)$. So if we define new equivalence relations $\mathcal S_n$ by $(x,y)\in\mathcal S_n$ if $(x,y)\in\mathcal R_n$ and $\forall m\geq n$, $\diam_\Phi([x]_{\mathcal R_m})\leq\varphi(n)$, we still have $\bigcup_{n\in\N} \mathcal S_n=\mathcal R$ and the $\mathcal S_n$-classes have a uniformly bounded diameter as wanted.

By Lemma \ref{lem: dense increasing union}, $\bigcup_{n\in\N}[\mathcal S_n]\cap [\Phi]_1'$ is dense in $[\Phi]_1'$. Then Lemma \ref{lem: bounded classes} yields that $[\mathcal S_n]\cap [\Phi]_1'=[\mathcal S_n]$ and that the induced topology on $[\mathcal S_n]$ is the uniform topology. Now recall that full groups of finite equivalence relations are whirly amenable for the uniform topology (they are products of groups of the form $\LL^0(Y,\nu,\mathfrak S_n)$ where $(Y,\nu)$ is non-atomic, and those are Levy groups by work of Y. Glasner, generalised by T. Giordano and V. Pestov \cite[Corollary 2.10]{Giordanoextremelyamenablegroups2007}). We conclude that each $[\mathcal S_n]$ is whirly amenable so that $[\Phi]_1'$ also is.
\end{proof}

In \cite[Theorem 5.1]{lemaitremeasurableanaloguesmall2018} it was shown that 
non amenable graphings have non amenable $\LL^1$ full groups and non amenable 
derived $\LL^1$ full groups. So Theorem \ref{thmi: amenable characterization} 
is established, and more generally the following holds.

\begin{theo}
Let $\Phi$ be an aperiodic graphing on a standard probability space $(X,\mu)$. 
The following are equivalent:
\begin{enumerate}[(1)]
\item the graphing $\Phi$ is amenable;
\item the $\LL^1$ full group $[\Phi]_1$ is amenable;
\item the derived $\LL^1$ full group $[\Phi]_1'$ is amenable.
\item the derived $\LL^1$ full group $[\Phi]_1'$ is extremely amenable.
\item the derived $\LL^1$ full group $[\Phi]_1'$ is whirly amenable. 
\end{enumerate}
\end{theo}

The above argument can be refined so as to show that derived $\LL^1$ full 
groups of amenable graphings are actually Lévy groups, and that one can find 
Levy families made of finite subgroups. The details are however tedious and the 
increasing chain of  finite subgroups is not explicit so we won't provide a 
proof here. It was asked in a first version of this paper whether the following 
metrics (which are induced by the $\LL^1$ metric on the $\LL^1$ full 
group of a $2$-odometer on the subgroups $\mathfrak S_{2^n}$) could provide a 
concrete Lévy 
family for the derived 
$\LL^1$ full group of the odometer.

\begin{defi}\label{df:L1 metric symmetric group}
Let $\mathfrak S_n$ be the symmetric group over $\{1,...,n\}$. Define 
\textbf{Spearman's footrule metric} on it by 
$$d^n_{\LL^1}(\sigma,\tau)=\frac 1n \sum_{i=1}^n \abs{\sigma(i)-\tau(i)}.$$
\end{defi}

Let us however point out that, as suggested by the referee, the family 
$(\mathfrak S_n,d^n_{\LL^1})$ is not a 
Lévy family. Indeed, by a result of Diaconis and Graham, the random variable 
$d^n_{\LL^1}(\id,\sigma)$ where $\sigma$ is a uniform element of $\mathfrak 
S_n$ has asymptotically mean $n/3$ and variance  $2n/45$ 
\cite[Thm. 1]{diaconisSpearmanFootruleMeasure1977} (moreover, suitably 
normalized, its 
law converges to the normal distribution). In particular, for $n$ large enough, 
the 
$\epsilon$-neighborhood of the ball of radius $n/3$ will be very close in 
measure to that same ball, thus violating concentration of measure.

\section{Large scale geometry and \texorpdfstring{$\LL^1$}{L1} full groups}\label{sec: geometry}

\subsection{Large scale geometry of Polish groups}

In this preliminary section, we recall how Rosendal's framework allow one to study Polish groups as geometric objects. First, a map $f:X\to Y$ between two metric spaces $(X,d_X)$ and $(Y,d_Y)$ is called a \textbf{quasi-isometric embedding} if there is a constant $C>0$ such that for all $x_1,x_2\in X$, 
$$-C+\frac 1C d_X(x_1,x_2)\leq d_Y(f(x_1),f(x_2))\leq C+ Cd_Y(f(x_1),f(x_2)).$$
It is \textbf{coarsely surjective} if there is $C>0$ such that every element of $Y$ is at distance at most $C$ from $f(X)$. Finally $f$ is a \textbf{quasi-isometry} if it is a coarsely surjective quasi-isometric embedding. We say that two metric spaces $(X,d_X)$ and $(Y,d_Y)$ are quasi-isometric if there is a quasi-isometry between them. It is not hard to check that quasi-isometry an equivalence relation between metric spaces. The equivalence class of a metric space is called its \textbf{quasi-isometry type}.

Two metrics $d$ and $d'$ on the same set $X$ are called \textbf{quasi-isometric} if the identity map is a quasi isometry between $(X,d)$ and $(X,d')$. Note that the coarse surjectivity condition is then automatic so this just means that there is $C>0$ such that for all $x_1, x_2\in X$, 
$$-C+\frac 1 C d(x_1,x_2)\leq d'(x_1,x_2)\leq C+ Cd(x_1,x_2).$$

A continuous right-invariant metric $d$ on a Polish group $G$ is called \textbf{maximal} if for every continuous right-invariant metric $d'$ on $G$, there is $C>0$ such that for all $g_1,g_2\in G$, $d'(g_1,g_2)\leq C+Cd(g_1,g_2)$. It follows from the definition that any two maximal right-invariant metrics on a Polish group are quasi-isometric, so if a Polish group admits such a metric, then it has a well-defined quasi-isometry type.

We finally give two conditions on a right-invariant metric which together are equivalent to the maximality of the  metric by \cite[Proposition 2.52]{rosendalCoarseGeometryTopological2018}.

\begin{defi}
A metric $d$ on a set $X$ is called \textbf{large-scale geodesic} if there is $K>0$ such that for any $x,y\in X$, there are $x_0,...,x_n\in X$ with $x_0=x$ and $x_n=y$ such that $d(x_i,x_{i+1})\leq K$ and
$$\sum_{i=0}^{n-1}d(x_i,x_{i+1})\leq K d(x,y).$$
\end{defi}

\begin{defi}
A right-invariant metric $d$ on a topological group $G$ is called \textbf{coarsely proper} if for every neighborhood of the identity $V$ and every $N\in\N$, there are a finite subset $F\subseteq G$ and $n\in\N$ such that 
$$B_d(e,N)\subseteq(FV)^n.$$
\end{defi}

\subsection{Derived \texorpdfstring{$\LL^1$}{L1} full groups  of amenable graphings}

In this section, we show that derived $\LL^1$ full groups of amenable graphings fit rather well into C. Rosendal's framework,  enabling their study as geometric objects. 

To be more precise, we show that the $\LL^1$ metric on derived $\LL^1$ full groups of amenable graphings is maximal among left-invariant continuous metrics up to quasi-isometry, which implies that derived $\LL^1$ full groups of amenable graphings have a well defined quasi-isometry type which is moreover induced by the $\LL^1$ metric. 

\begin{prop}
Let $\Phi$ be an amenable graphing. Then every element of $[\Phi]_1'$ is a limit of periodic elements. 
\end{prop}
\begin{proof}
Using the Connes-Feldmann-Weiss theorem \cite{connesamenableequivalencerelation1981a}, we write $\mathcal R_\Phi$ as an increasing union of finite equivalence relations $\mathcal R_n$. Then every element of $[\mathcal R_n]$ is periodic and by Lemma \ref{lem: dense increasing union} the union $\bigcup_{n\in\N}[\mathcal R_n]\cap [\Phi]_1'$
is dense in $[\Phi]_1'$ so we are done.
\end{proof}


The next lemma is key to our results on the geometry of $\LL^1$ full groups. 
\begin{lemm}\label{lem:product of small}
	Let $\Phi$ be a graphing. Let $T\in[\Phi]_1$ be periodic. Then for every $N\in\N$, there are periodic elements $V_1,...,V_N\in [\Phi]_1$ such that $T=V_1\cdots V_N$ and for every $i\in\{1,...,N\}$, 
	$$d^1_\Phi(V_i,\id_X)=\frac 1 N d^1_\Phi(T,\id_X).$$
\end{lemm}
\begin{proof}
Let $M=d^1_\Phi(T,\id_X)$.
Let $(A_t)_{t\in[0,M]}$ be a continuous increasing path from $\emptyset$ to $\supp T$, and consider $B_t=\bigcup_{n\in\N}(TU\inv)^n(A_t)$. Since $TU\inv$ is periodic, the map $t\mapsto B_t$ is still a continuous increasing path from $\emptyset$ to $\supp T$ now consisting of $T$-invariant sets.
 
Define a continuous increasing map $\psi:[0,M]\to[0,M]$ by  $$\psi(t)=\int_{B_t}d_\Phi(x,T(x)).$$ We have $\psi(0)=0$ and $\psi(M)=M$, so  by the intermediate value theorem we find $t_0=0<t_1<...<t_{N-1}<t_N=M$ such that for all $i\in\{0,...,N\}$, $\psi(t_i)=\frac{iM}N$. We then let $A_i=B_{t_i}\setminus B_{t_{i-1}}$ for each $i\in\{1,...,N\}$. By construction each $A_i$ is $T$-invariant and $\supp T=\bigsqcup_{i=1}^n A_i$. So if we define $V_i=T_{A_i}$, we get $T=\prod_{i=1}^n V_i$. Finally for each $i\in\{1,...,n\}$ the equality $A_i=B_{t_i}\setminus B_{t_{i-1}}$ yields $$d^1_\Phi(\id_X,V_i)=\int_{B_{t_i}}d_\Phi(x,T(x))-\int_{B_{t_{i-1}}}d_\Phi(x,T(x))=\psi(t_{i})-\psi(t_{i-1})=\frac MN$$
as wanted.
\end{proof}

\begin{rema}
	Observe that the above proof also yields the existence of a geodesic 
	segment between any periodic element and the identity map. Indeed, if 
	$\psi\inv$ is a right inverse of $\psi$ defined above, we let 
	$W_t=T_{B_{\psi\inv(t)}}$, and then $(W_t)_{t\in[0,M]}$ is a geodesic 
	segment from $\id_X$ to $T$.
\end{rema}

\begin{prop}
	Let $\Phi$ be an amenable graphing, let $K>0$ and let $N\in\N$. Then we have the following inclusions in the derived $\LL^1$ full group of $\Phi$:
	$$B_{d^1_\Phi}(1,K/N)^N\subseteq B_{d^1_\Phi}(1,K)\subseteq B_{d^1_\Phi}(1,K/N)^{N+1}.$$
\end{prop}
\begin{proof}
	The first inclusion is a straightforward consequence of the triangle inequality and right invariance. For the second one, let $T\in B_{d^1_\Phi}(1,K)$, then since periodic elements are dense we can find $U\in B_{d^1_\Phi}(1,K/N)$ such that $TU\inv$ is periodic and $TU\inv\in B_{d^1_\Phi}(\id_X,K)$. 
	
	By the previous proposition, the are $V_1,...,V_N\in B_{d^1_\Phi}(\id_X,K/N)$ such that $TU\inv =V_1\cdots V_n$. We conclude that $T=TU\inv U=\left(\prod_{k=1}^NV_k\right) U$ so $T\in B_{d^1_\Phi}(1,K/N)^{N+1}$ as wanted. 
\end{proof}

\begin{coro}\label{cor: l1 is maximal on derived of amenable}
Let $\Phi$ be an amenable graphing, then the $\LL^1$ metric on the derived $\LL^1$ full group of $\Phi$ is maximal.
\end{coro}
\begin{proof}
By the previous proposition, the $\LL^1$ metric is both large-scale geodesic and coarsely proper. So \cite[Proposition 2.52]{rosendalCoarseGeometryTopological2018} yields that the $\LL^1$ metric is maximal. 
\end{proof}

We finally show that although derived $\LL^1$ full groups of amenable graphings do have a geometry, the latter is fundamentally infinite-dimensional.

\begin{prop}\label{prop:infinite dimensional}
Suppose $\Gamma$ is an amenable group containing an element of infinite order, and let $\Gamma\act(X,\mu)$ be a free measure-preserving action. Then for every $n\in\N$, the derived $\LL^1$ full group of the $\Gamma$ action contains a quasi-isometric copy of $\R^n$. 
\end{prop}
\begin{proof}
It suffices to produce a quasi-isometric copy of $\Z^n$ since the latter is quasi-isometric to $\R^n$. 
Let $T$ be the aperiodic measure-preserving transformation which corresponds to the element of $\Gamma$ of infinite order. 
Using Rokhlin's lemma for $T$, we find $A\subseteq X$  which intersect almost every $T$-orbit and such that all sets $A,T(A),...,T^{2n-1}(A)$ are disjoint. 
For each $i\in\{0,...,n-1\}$ let $T_i=T_{T^{2i}(A)}T_{T^{2i+1}(A)}\inv$.  
Observe that since $T{T_{T^{2i}(A)}}T\inv=T_{T^{2i+1}(A)}$, each $T_i$ is a commutator and thus belongs to the derived $\LL^1$ full group. 
It is now straightforward to check that the marked group generated by $T_1,...,T_n$ is quasi-isometric to $\Z^n$. 
\end{proof}

As pointed out by the referee, it would be interesting to know whether some 
infinite dimensional metric spaces such as $\ell^1$ can be coarsely embedded 
into derived $\LL^1$ full groups.
We end this section with a basic open question on the geometry of derived 
$\LL^1$ full groups.

\begin{ques}\label{qu: quasi-isom 2 3 odometer}
Are the derived $\LL^1$ full groups of the $2$-odometer and of the $3$-odometer quasi-isometric ?
\end{ques}

Note that by \cite{dowerkBoundedNormalGeneration2019}, the derived $\LL^1$ full 
groups of the $2$-odometer and of the $3$-odometer are not isomorphic as 
abstract groups. Also, the countable group of dyadic  permutations $\mathfrak 
S_{2^\infty}$ (resp. $\mathfrak S_{3^\infty}$) is dense in the derived $\LL^1$ 
full group of the $2$-odometer (resp. $3$-odometer), and the induced metric is 
the inductive limit of the metrics from Definition \ref{df:L1 metric symmetric 
group}, so this is really a question about quasi-isometry for the non finitely 
generated countable groups $\mathfrak S_{2^\infty}$ and $\mathfrak 
S_{3^\infty}$ equipped with these natural metrics. 

\subsection{\texorpdfstring{$\LL^1$}{L1} full groups of ergodic \texorpdfstring{$\Z$}{Z}-actions} 

For measure-preserving $\Z$-actions, we have a much better understanding of elements of the $\LL^1$ full group. This will allow us to show that the $\LL^1$ metric on the whole $\LL^1$ full group is maximal, and hence defines its quasi isometry type. Since a measure-preserving $\Z$-action is nothing but a single measure-preserving transformation, we rather speak of $\LL^1$ full groups of measure-preserving transformations. Let us start by reminding Belinskaya's decomposition results for elements of the $\LL^1$ full group \cite{BelinskayaPartitionsLebesguespace1968}. 

Given a measure-preserving transformation $T$, we denote by $[T]_1$ its $\LL^1$ full group. If $T$ is aperiodic then every $T$-orbit is endowed with a linear ordering $\leq_T$ defined by $x\leq_Ty$ if and only if there is $n\geq 0$ such that $y=T^n(x)$. An element $U$ of the full group of $T$ is \textbf{almost positive} if for all $x\in X$, there is $N\in\N$ such that for all $n\geq N$ we have $U^n(x)\geq_Tx$. It is \textbf{almost negative} if for all $x\in X$, there is $N\in\N$ such that for all $n\geq N$ we have $U^n(x)\leq_T x$. Every element of the $\LL^1$ full group has a natural decomposition with respect to such elements. 

\begin{lemm}\label{lem:decompose positive negative periodic}
Let $T$ be a measure-preserving aperiodic transformation, let $U\in [T]_1$. Then there is a partition $X=X_-\sqcup X_{per}\sqcup X_+$ of $X$ into $U$-invariant sets such that $U_{X_-}$ is almost negative, $U_{X_+}$ is almost positive and $U_{X_{per}}$ is periodic. 
\end{lemm}
\begin{proof}
	This is a direct consequence of \cite[Theorem 3.2]{BelinskayaPartitionsLebesguespace1968}.
\end{proof}

Observe that since the sets from the above decomposition are $U$-invariant and partition $X$, we then have $U=U_{X_-}U_{X_{per}}U_{X_+}$ and
$$d^1_T(\id_X,U)=d^1_T(\id_X,U_{X_-})+
d^1_T(\id_X,U_{X_{per}})+d^1_T(\id_X,U_{X_+}).$$
Recall that if $T$ is an aperiodic measure-preserving transformation, every $U\in[T]_1$ has an associated cocycle $c_U:X\to \Z$ uniquely defined by the equation $U(x)=T^{c_U(x)}(x)$. We can then defined the \textbf{index map} $I:[T]_1\to\R$ by $I(U)=\int_X c_U(x)d\mu(x)$. In \cite[Corollary 4.20]{lemaitremeasurableanaloguesmall2018}, we prove that if $T$ is ergodic, the kernel of the index map is the derived $\LL^1$ full group of $T$ and the index map takes values into $\Z$. The conjunction of Proposition 4.15 and Proposition 4.17 from \cite{lemaitremeasurableanaloguesmall2018} implies the following result.

\begin{lemm}
Let $U\in [T]_1$ be almost positive, let $n=I(U)$. Then there is a family $(A_i)_{i=1}^n$ of subsets of $X$  such that $U=T_{A_1}\cdots T_{A_n} V$, where $V$ is periodic. \qed
\end{lemm}

\begin{theo}
Let $T$ be an ergodic measure-preserving transformation.
Then the $\LL^1$ metric on the $\LL^1$ full group of $T$ is maximal. 
\end{theo}
\begin{proof}
We first show that for every $N\in\N$, every element of $[T]_1$ of $\LL^1$ norm\footnote{By $\LL^1$ norm, we mean the $\LL^1$ distance to $\id_X$.} less than $N$ can be written as a product of at most $11N$ elements, all of which are either $T$, $T\inv$, or periodic of $\LL^1$ norm less than $1$. 

Let $X=X_{per}\sqcup X_+\sqcup X_-$ be the partition of $X$ into $U$-invariant sets provided by Lemma \ref{lem:decompose positive negative periodic}. We start by noting that $d^1_T(U_{X_+},\id_x)\leq d^1_T(U,\id_x)$ and thus $I(U_{X_+})<N$. We may thus write
$$U_{X_+}=T_{A_1}\cdots T_{A_n} V,$$
where $n=I(U_{X_+})<N$ and $V$ is periodic. By Kac's return time theorem \cite[Theorem 2']{Kacnotionrecurrencediscrete1947}, each $T_{A_i}$ is at distance $1$ from $\id_X$, so by the triangle inequality, we have $d^1_T(V,\id_X)<2N$. 

Now for each $i\in \{1,...,n\}$ we have $T_{A_i}=(T_{A_i}T\inv)T$. Moreover 
$T_{A_i}T\inv$ is periodic and $d^1_T(T_{A_i}T\inv,\id_X)<2$. Since $U_{X_+}$ 
is equal to the product $(T_{A_1}T\inv)T\cdots (T_{A_n}T\inv)T V$ and 
$d^1_T(V,\id_X)<2N$, Lemma \ref{lem:product of small} applied to both $V$ and 
the $T_{A_i}T\inv$ yields that $U_{X_+}$ is a product of at most $5N$ elements, 
all of which are either periodic and in the open unit ball or equal to $T$. 

The same reasoning for $U_{X_-}$ yields that $U_{X_-}$ can be written as a product of at most $5N$ elements, all of which are either periodic of norm $<1$ or equal to $T\inv$.

Finally, Lemma \ref{lem:product of small} yields that $U_{X_{per}}$ can be written as the product of at most $N$ periodic elements of norm $<1$. 
We can thus conclude that $U=U_{X_+}U_{X_{per}}U_{X_-}$ can be written as the product of at most $11N$ elements, all of which are either periodic of norm $<1$ or belong to $\{T,T\inv\}$.

This fact already yields that the $\LL^1$ metric is large-scale geodesic, and we now need to show that $d^1_T$ is coarsely proper. 
So let $N\in\N$, let $V$ be a neighborhood of the identity.
We find $k\in\N$ such that $B_{d^1_T}(\id_X,1/k)\subseteq V$ and we deduce from the above fact along with  Lemma \ref{lem:product of small} that
$$B_{d^1_T}(\id_X,N)\subseteq(\{T,T\inv\} V)^{11kN}.$$
The metric $d^1_T$ is thus both large-scale geodesic and coarsely proper, so using \cite[Proposition 2.52]{rosendalCoarseGeometryTopological2018} we conclude that $d^1_T$ is maximal as wanted. 
\end{proof}

\subsection{Automatic quasi-isometry: proof of Theorem \ref{thmi: auto QI}}\label{sec: automatic QI}

In this section we prove the following result, which directly yields Theorem \ref{thmi: auto QI}.
\begin{theo}\label{thm: auto QI}
	Let $\Phi$ and $\Psi$ be two amenable ergodic graphings on a standard probability space  $(X,\mu)$. Then every abstract group isomorphism between their $\LL^1$ full groups must be a quasi-isometry for their respective $\LL^1$ metrics. 
\end{theo}
\begin{proof}
	Let $\rho: [\Phi]_1\to[\Psi]_1$ be an abstract group isomorphism. 
	First observe that by \cite[Corollary 3.19]{lemaitremeasurableanaloguesmall2018}, there is a measure-preserving transformation $S$ such that for all $\U\in[\Phi]_1$, $\rho(U)=SUS\inv$. In other words, we can assume that $\rho$ is the conjugacy by some $S\in\Aut(X,\mu)$, and it follows that $\rho$ is $d_u$-isometric, in particular it is continuous for the uniform topology. 
	On $[\Phi]_1$ the uniform topology is refined by the $\LL^1$ topology, and since it separates points we get that it generates the Borel $\sigma$-algebra of $[\Phi]_1$ induced by the $\LL^1$ topology. We conclude that $\rho$ is a Borel group isomorphism, and it follows from Pettis' lemma (see e.g. \cite[Theorem 1.2.5]{beckerDescriptiveSetTheory1996}) that $\rho$ is a topological group isomorphism. 
	
	In particular, $\rho$ induces a topological group isomorphism between $[\Phi]_1'$ and $[\Psi]_1'$, and since their respective $\LL^1$ metrics are maximal (Corollary \ref{cor: l1 is maximal on derived of amenable}), we conclude that $\rho$ induces a quasi-isometry for their respective $\LL^1$ metrics: there is $C>0$ such that for all $T_1,T_2\in[\Phi]'_1$, 
	$$-C+\frac 1Cd^1_\Phi(T_1,T_2)\leq d^1_\Psi(\rho(T_1),\rho(T_2))\leq Cd^1_\Phi(T_1,T_2)+C$$
	
	In order to establish the same inequality for elements of the whose $\LL^1$ full group, first note that by right-invariance, one may as well assume $T_1=\id_X$, and then note that by symmetry, it suffices to find $C'>0$ such that for all $T\in[\Phi]_1$, 
	$$d^1_\Phi(\id_X,\rho(T)\leq C'd^1_\Psi(\id_X,T)+C'.$$
	We will show that $C'=6C$ works. 
	
	Let $T\in[\Phi]_1$, then by a well-known lemma (see e.g. \cite[Proposition 2.7]{lemaitremeasurableanaloguesmall2018}), there is a partition $(A_1,A_2,B_1,B_2,B_3)$ of $\supp T$ such that 
	$$T(A_1)=A_2, T(B_1)=B_2\text{ and } T(B_2)=B_3.$$
	We use this decomposition to define three involutions $U_1,U_2,U_3$ as follows:
	\begin{align*}
	U_1(x)&=\left\{\begin{array}{cl}
		T(x) & \text{if }x\in A_1\sqcup B_1\\
		T\inv(x) & \text{if }x\in A_2\sqcup B_2\\
		x & \text{otherwise,}
	\end{array}
	\right.\\
	U_2(x)&=\left\{\begin{array}{cl}
		T(x) & \text{if }x\in B_2\\
		T\inv(x) & \text{if }x\in B_3\\
		x & \text{otherwise and}
	\end{array}
	\right.
\\
	U_3(x)&=\left\{\begin{array}{cl}
		T(x) & \text{if }x\in A_2\sqcup B_3\\
		T\inv(x) & \text{if }x\in T(A_2)\sqcup T(B_3)\\
		x & \text{otherwise.}
	\end{array}
	\right.
	\end{align*}
	By construction for each $i=1,2,3$ we have $d^1_\Phi(\id_X,U_i)\leq 2d^1_\Phi(\id_X,T)$, moreover since $U_1,U_2$ and $U_3$ are involutions they actually belong to the derived $\LL^1$ full group of $\Phi$. We can then decompose $T$ as 
	$$T=U_{1\restriction A_1\sqcup B_1}\sqcup U_{2\restriction B_2}\sqcup U_{3\restriction A_2\sqcup B_3}\sqcup\id_{X\setminus\supp T}.$$
	Because $\rho$ is the conjugacy by the transformation $S$, we then have 
	$$\rho(T)=\rho(U_1)_{\restriction S(A_1\sqcup B_1)}
			\sqcup \rho(U_2)_{\restriction S(B_2)}
			\sqcup \rho(U_3)_{\restriction S(A_2\sqcup B_3)}
			\sqcup\id_{X\setminus S(\supp T)}.$$
	We thus have $d^1_\Psi(\id_X,\rho(T))\leq \sum_{i=1}^3d^1_\Psi(\id_X,U_i)$.
	Since for $i=1,2,3$, $$d^1_\Psi(\id_X,\rho(U_i))\leq C d^1_\Phi(\id_X,U_i)+C\leq 2Cd^1_\Phi(\id_X,T)+C,$$
	we conclude that 
	$$d^1_\Psi(\id_X,\rho(T))\leq \sum_{i=1}^3d^1_\Psi(\id_X,U_i)\leq 6C d^1(\Phi(\id_X,T))+3C$$
	as wanted.
\end{proof}

Note that if we knew that the $\LL^1$ metric on $\LL^1$ full groups of amenable graphings is maximal, then the first two paragraphs of the above proof would suffice to get the conclusion of Theorem \ref{thm: auto QI}. As explained in the introduction, we actually conjecture that the following is true, which would then imply that one can remove both the amenability and the ergodicity hypothesis in Theorem \ref{thm: auto QI} . 

\begin{conj}
	Let $\Phi$ be a graphing. Then the $\LL^1$ metric on the $\LL^1$ full group of $\Phi$ is maximal. 
\end{conj}

We end this section by explaining why in Theorem \ref{thm: auto QI} the conclusion cannot be strengthened to bilipschitzness. Recall that in a topological space, a subset is \textbf{$G_\delta$} if it can be written as a countable intersection of open sets, and that by the Baire category theorem, in a complete metric space every countable intersection of \emph{dense} open sets has to be dense.

\begin{lemm}\label{lem: unbounded return}
	Let $T$ be an aperiodic transformation. Then the set of all $A\in\MAlg(X,\mu)$ such that the return time to $A$ is unbounded is dense $G_\delta$ in $\MAlg(X,\mu)$.
\end{lemm}
\begin{proof}
	The set $\mathcal B$ of all $A\in\MAlg(X,\mu)$ such that the return time to $A$ is unbounded is the intersection over $n\in\N$ of the sets 
	$$\mathcal B_n:=\{A\in \MAlg(X,\mu): \mu(T^{-n}(A)\setminus \bigcup_{i=0}^{n-1}T^{-i}(A))>0\},$$
	so it is $G_\delta$. Let us now fix $n\in \N$ and show that $\mathcal B_n$ is dense so as to finish the proof. Let $\epsilon>0$, let $N\in\N$ such that $\frac nN<\epsilon$. By aperiodicity we find $B\subseteq X$ with positive measure such that $B,T(B),...,T^{Nn}(B)$ are disjoint, and so in particular $\mu(B)<\frac1{Nn}$. Now observe that for any $A\in\MAlg(X,\mu)$, the set 
	$$(A\cup B\cup T^n(B))\setminus \bigsqcup_{i=1}^{n-1}T^i(B)$$
	is in $\mathcal B_n$ and is $n\mu(B)<\epsilon$-close to $A$, which shows that $\mathcal B_n$ is dense.
\end{proof}

\begin{exam}\label{ex: not bilip iso} 
	Let $T$ be an aperiodic measure-preserving transformation, and consider $A\subseteq X$ with unbounded return time, as provided by the previous lemma. Then the conjugacy by the induced transformation $T_A$ is not bilipschitz: for each $n$ we may find $B\subseteq A$ such that  $T(B)$, ..., $T^n(B)$ are all disjoint from $A$, and then if $U$ is the involution defined by $U=T_{\restriction B}\sqcup T\inv_{\restriction T(B)}\sqcup \id_{X\setminus (B\sqcup T(B)}$, we have $d^1_T(\id_X,T_AUT_A\inv)\geq n d^1_T(\id_X,U)$.	
\end{exam}

\section{Topological rank two for rank one transformations}\label{sec: rank one}

In this section we prove that every rank one measure-preserving transformation has an $\LL^1$ full group whose topological rank is equal to two. Let us recall the definition of such transformations and first introduce some useful terminology.

\begin{defi}\label{df: Rokhlinize}
Let $T$ be a measure-preserving transformation and $A\subseteq X$ Borel. Let $N_T(A)\in\N$ be the greatest integer such that the sets $A, T(A),..., T^{N_T(A)-1}(A)$ are all disjoint. We then have a \textbf{Rokhlin partition} $\mathcal P_{A,T}=(A,T(A),...,T^{N_T(A)-1}(A))$ of $Y_{A,T}:=\bigcup_{i=0}^{N_T(A)-1}T^i(A)$ and we let $M_{A,T}$ be the sets of finite unions of elements of $\mathcal P_{A,T}$. 
\end{defi}

Note that Rokhlin's lemma guarantees the existence of Rokhlin partitions $\mathcal P_{A,T}$ with arbitrarily small atoms and $\mu(Y_{A,T})$ arbitrarily close to $1$. 

\begin{defi}
A measure-preserving transformation $T$ is \textbf{rank one} if for every $\epsilon>0$ and every finite family $(B_i)_{i=1}^k$ of Borel sets there is a Borel set $A$  and Borel sets $A_1,...,A_k\in M_{A,T}$ such that for all $i\in\{1,...,k\}$, 
$$\mu(A_i\bigtriangleup B_i)<\epsilon.$$
\end{defi}

It is not hard to check from the definition that rank one transformations are ergodic and form a $G_\delta$ set in $\Aut(X,\mu)$ for the weak topology (see \cite[Lemme 5.26]{lemaitreGroupesPleinsPreservant2014} for a detailed proof). An easy example of a rank one transformation is provided by the odometer, and so by Halmos' conjugacy lemma (see \cite[Theorem 2.4]{KechrisGlobalaspectsergodic2010}), rank one transformations form a dense $G_\delta$ set in $\Aut(X,\mu)$.

The following theorem of Baxter says that rank one transformations are actually quite close to odometers. 

\begin{theo}[\cite{Baxterclassergodictransformations1971}]Let $T$ be a rank one 
transformation. Then there is a decreasing sequence $(A_n)$ of Borel subsets of 
$X$ such that $M_{A_n,T}$ is  \emph{increasing} and its union is $\mu$-dense in 
the Borel $\sigma$-algebra. \end{theo}

A sequence $(A_n)$ as in the previous theorem will be called a \textbf{rank one basis} of $T$, so in our terminology a measure-preserving is rank one if and only if it admits a rank one basis. Observe that if $(A_n)$ is a rank one basis for $T$ then the measure of $A_n$ then tends to zero and $N_T(A_n)\to+\infty$.

\begin{exam}	
Let $T_0$ be the dyadic odometer on $X=\{0,1\}^\N$, and for every $n\in\N$ we let $A_n$ denote the set of sequences $(x_i)\in \{0,1\}^\N$ such that $x_0=\cdots=x_n=0$. Then $(A_n)$ is a rank one basis for $T_0$ with $N_{T_0}(A_n)=2^n$ and the additional property that $Y_{A_n,T}=X$ for every $n\in\N$.
\end{exam}

We will now generalize to rank one transformations the fact that dyadic transformations are dense in the derived $\LL^1$ full group of the odometer. Let us first see how Definition \ref{df: Rokhlinize} provides us natural embeddings of the symmetric group in the derived $\LL^1$ full group.

Let $T$ be a measure-preserving transformation and $A\subseteq X$ be Borel. We then define the embedding 
$$\rho_A: \mathfrak S(\{0,...,N_T(A)-1\})\to [T]_1'$$
by: for all $\sigma\in\mathfrak S(\{0,...,N(A)-1\})$ and all $x\in Y_{A,T}$, if $i\in\{0,...,N(A)-1\}$ is such that $x\in T^i(A)$, then 
$$\rho_A(\sigma)(x)=T^{\sigma(i)-i}(x),$$
and $\rho_A(\sigma)(x)$ for $x\not\in Y_{A,T}$. Observe that by construction for all $i\in\{0,...,N(A)-1\}$ we have $\rho_A(\sigma)(T^i(A))=T^{\sigma(i)}(A)$.
We then  let $$\displaystyle\mathfrak S_{A,T}:=\rho_A(\mathfrak S(\{0,...,N(A)-1\}))\text{ and }\mathfrak A_{A,T}:=\rho_A(\mathfrak A(\{0,...,N(A)-1\})).$$
Observe that if $A\subseteq B$ and $M_{A,T}\subseteq M_{B,T}$ then $\mathfrak S_{A,T}\leq \mathfrak S_{B,T}$ and $\mathfrak A_{A,T}\leq\mathfrak A_{B,T}$. 
The following generalizes to $\LL^1$ full groups a result from the thesis of the author \cite[Lemme 5.31]{lemaitreGroupesPleinsPreservant2014}.

\begin{prop}\label{prop: rank one alternating is dense in derived}
Let $(A_n)$ be a rank one basis of a measure-preserving transformation $T$. Then 
$$\bigcup_{n\in\N}\mathfrak A_{A_n,T}$$
is dense in $[T]_1'$. 
\end{prop}
\begin{proof}
Let $U_n$ be the involution defined by $U_n=\varphi_{A_n}((0\; 1))$, then $d^1_T(U_n,1)\to 0$. Since  $\mathfrak S_{A_n,T}=\mathfrak A_{A_n,T}\cup U_n\mathfrak A_{A_n,T}$, we conclude that  the density of $\bigcup_{n\in\N}\mathfrak A_{A_n,T}$ will follow from that of $\bigcup_{n\in\N}\mathfrak S_{A_n,T}$ which will thus prove instead. 

Since $[T]_1'$ is topologically generated by involutions of the form $I_{T,A}$ (cf. Theorem \ref{thm: topo gen by I_phi}), we only need to approximate such involutions by elements of $\bigcup_{n\in\N}\mathfrak S_{A_n,T}$. So let $A\subseteq X$, we then have a sequence $(B_n)$ of elements of $M_{A_n,T}$ with $B_n\to A$. Furthermore we may assume that $B_n$ is disjoint from $A_n\sqcup T^{N(A_n)-1}$ since the measure of $A_n\sqcup T^{N(A_n)-1}$ tends to zero as $n$ tends to infinity. This implies $I_{B_n,T}\in \mathfrak S_{A_n,T}$ and since  $I_{B_n,T}\to I_{A,T}$ we get the desired conclusion. 
\end{proof}

%

\begin{lemm}\label{lem: generating alternating group}
Let $T$ be a measure preserving transformation, let $A\subseteq X$ and let $n\in\N$ such that $3\leq n \leq N_T(A)-2$. Consider the element $U=\rho_A((0\quad1\quad\cdots\quad n-1))$. Then the group generated by $T$ and $U$ contains $\mathfrak A_{A,T}$. 
\end{lemm}
\begin{proof}
By conjugating $U$ by powers of $T$, we see that the group generated by $T$ and $U$ contains all the 
$\rho_A((i\quad i+1\quad\cdots\quad i+n-1))$ for $i+n-1\leq N(A)$. The lemma then follows from the fact that $\mathfrak A(\{0,...,N(A)-1\})$ is contained in the group generated by the $(i\quad i+1\quad\cdots\quad i+n-1)$ for $0\leq i\leq N(A)-n$. 
We briefly recall the proof of the latter fact for the reader's convenience. 

Denote by $G$ the group generated by the $(i\quad i+1\quad\cdots\quad i+n-1)$ for $0\leq i\leq N(A)-n$. First note that $G$ contains the permutation $$(n\quad n-1\quad\cdots\quad 1)(0\quad 1\cdots\quad  n-1)=(0\quad n \quad n-1)$$ so it also contains $(0\quad n \quad n-1)^2=(0\quad n-1\quad n)$. Conjugating by a power of $(1\quad 2\quad\cdots\quad n)$, we obtain that $G$ contains the element $(0\quad 1\quad 2)$. Conjugating again by appropriate cycles we get that $G$ contains all the $(i\quad i+1\quad i+2)$ for $0\leq i\leq n-3$, and since the latter generate the group $\mathfrak A(\{0,...,N(A)-1\})$ the result follows.  
\end{proof}

The usual way of disjointifying a countable family of sets $(B_n)$ is to let $B'_n:=B_n\setminus \bigcup_{m<n}B_m$, but note that letting $B'_n=B_n\setminus \bigcup_{m>n}B_m$ also works. We will need to apply this idea to measure-preserving transformations so as to make their support disjoint. 

For $p\in\N$, let us call a \textbf{$p$-cycle} a transformation all whose orbits have cardinality either $1$ or $p$. Given a countable family $(U_n)_{n\in\N}$ where each $U_n$ is a $p_n$-cycle, we explain how to change them so that they have disjoint support. For each $n$, we  let $V_n$ be the transformation induced by $U_n$ on the biggest $U_n$-invariant set contained in
$\supp U_n\setminus \bigcup_{m>n}\supp U_m$.
Note that since $U_n$ is a $p_n$-cycle, we can easily compute the support of $V_n$ which is
$$\supp V_n=\supp U_n\setminus \left(\bigcup_{i=0}^{p_n-1}U_n^i\left(\bigcup_{m>n}\supp U_m\right)\right).$$
Of course this support might be empty, but observe that if the supports of the remaining $U_m$'s are very small then $V_n$ will actually be very close to $U_n$. The sequence $(V_n)$ obtained this way is called the \textbf{disjointification} of the sequence $(U_n)$.

We are now ready to prove the main result of this section.

\begin{theo}
Let $T$ be a rank one transformation. Then the topological rank of its $\LL^1$ full group $[T]_1$ is equal to $2$. 
\end{theo}
\begin{proof}
Let $(p_n)_{n\in\N}$ be the increasing enumeration of prime numbers. Let $\epsilon_n$ be a sequence of positive real numbers decreasing to zero, e.g. $\epsilon_n=2^{-n}$. Let $(A_n)$ be a one-dimensional ranking of $T$.

Using lemma \ref{lem: generating alternating group}, one can now build by induction an increasing sequence  of integers $(k_n)$ and positive real numbers $(\delta_n)$ such that for all $n$, if we let $$U_n:=\varphi_{A_{k_n}}((0\quad 1\cdots p_n-1))$$ then the following conditions are satisfied.
\begin{enumerate}[(i)]
\item \label{cond: delta_n} Any element $V_n$ which is $\delta_n$-close to $U_n$ satisfies that the group generated by $T$ and $V_n$ contains every element of $\mathfrak A_{A_{k_n,T}}$ up to an $\epsilon_n$ error.
\item \label{cond: V_n approximate U_n} For  every $m<n$, we have $\displaystyle p_m\sum_{k=m+1}^{n}\mu(\supp U_k)< \delta_m$.
\item \label{cond: we can appply infinite product trick} We have $k_n\geq 4^{n2^n+2^n}$.
\end{enumerate}

Consider the sequence $(V_n)$ defined as the disjointification of the $U_n$'s. Thanks to condition (\ref{cond: V_n approximate U_n}) we have that $V_n$ is $\delta_n$-close to $U_n$. By condition (\ref{cond: delta_n}) the group $\la T,V_n\ra$ contains every element of $\mathfrak A_{A_{k_n,T}}$ up to an $\epsilon_n$ error.

Now let $V:=\prod_{n\in\N} V_n$. We first establish a general inequality which will yield that $V\in[T]_1$ and moreover the closed group generated by $V$ contains every $V_n$. 

\begin{claim}
For all $n\in\N$ we have \begin{align}\label{ineq: prod trick works}
\left( \prod_{i<n}p_i \right)\left(\sum_{i\geq n}\frac{p_i}{k_i}\right)\leq2^{-n2^n}.
\end{align}
\end{claim}
\begin{claimproof}
We first need an explicit bound on $p_n$. A well-known way to get such a bound is as follows: Euclid's argument that there are infinitely many prime numbers actually shows that  $p_{n+1}\leq p_0\cdots p_n+1$. From this, a straightforward induction yields the bound $$p_n\leq 2^{2^n}.$$
Now by condition (\ref{cond: we can appply infinite product trick}) we have $k_n\geq 4^{n2^{n}+2^n}$, so for every $n\geq 1$ we have 
\begin{align}\label{ineq: disjointness trick can be carried}
 \sum_{i\geq n}\frac{p_i}{k_i} & \leq 4^{-n2^n} \sum_{i\geq n}\frac{2^{2^{i}}}{4^{2^i}}\\
 &\leq 4^{-n2^n}
\end{align}
Now $\prod_{i<n}p_i\leq (p_n)^n\leq 2^{n2^n}$ so combining this with the previous inequality we see that \eqref{ineq: prod trick works} holds.
\end{claimproof}
Now observe that for every $n\in \N$ we have $d^1_T({V_n},\id_X)_1\leq d^1_T(U_n,\mathrm{id}_X)\leq\frac{2p_n}{k_n}$. The inequality \eqref{ineq: prod trick works} for $n=0$ yields $\sum_{i\geq 0}\frac{2p_i}{k_i}<+\infty$ so that $V\in[T]_1$ as announced.

Fix $n\in\N$. For every $m\geq n+1$ let $P_m=\prod_{i< m, i\neq n} p_i$ and let $t_m\in\{1,...,p_n-1\}$ such that $P_mt_m=1\mod p_n$. Now observe that
$$V^{P_mt_m}=V_n\prod_{i\geq m}V_i^{P_mt_m}.$$

Since $t_m\leq p_n$, we have $P_mt_m\leq\prod_{i< m} p_i$ so that  
$$d^1_T\left(\prod_{i> m}V_i^{P_mt_m},\id_X\right)\leq \left( \prod_{i< m}p_i \right)\sum_{i\geq k}d^1_T(V_i,\id_X).$$
But for every $i$ we have $\displaystyle\norm{V_i}_1\leq \frac{2p_i}{k_i}$ so we conclude by inequality \eqref{ineq: prod trick works} that we have $d^1_T(\prod_{i> m}V_i^{P_mt_m},\id_X)\to 0$ and hence $$V^{P_mt_m}\to V_n\quad[m\to+\infty].$$
So for every $n\in \N$ the closed subgroup generated by $V$ contains $V_n$. 

Since the group $\la T,V_n\ra$ contains every element of $\mathfrak A_{A_{k_n,T}}$ up to an $\epsilon_n$ error and $\epsilon_n\to 0$, we conclude by Proposition \ref{prop: rank one alternating is dense in derived} that the group generated by $T$ and $V$ is dense in $[T]_1$ which has thus topological rank at most $2$. Since the $\LL^1$ full group $[T]_1$ is not abelian, we conclude that $[T]_1$ has topological rank 2 as wanted. 
\end{proof}

\section{Further remarks and questions}

\subsection{\texorpdfstring{$\LL^p$}{Lp} orbit equivalence and \texorpdfstring{$\LL^p$}{Lp} full orbit equivalence} \label{sec: lp full OE}

It is  an instructive exercise to show that given two measure preserving actions $\alpha:\Gamma\to\Aut(X,\mu)$, $\beta:\Lambda\to\Aut(X,\mu)$ and a measure-preserving transformation $T$, the following are equivalent:
\begin{enumerate}[(i)]
	\item $T$ sends $\alpha(\Gamma)$-orbits to $\beta(\Lambda)$-orbits ;
	\item $T\alpha(\Gamma)T\inv$ is a subgroup of  $[\beta(\Lambda)]$ and 
	$T\inv\beta(\Lambda)T$ is a subgroup of $[\alpha(\Gamma)]$;
	\item $T[\alpha(\Gamma)]T\inv=[\beta(\Gamma)]$.
\end{enumerate}
If $T$ satisfies any of the above three conditions, one says that $T$ is an \textbf{orbit equivalence} between the actions $\alpha$ and $\beta$, and that the latter are \textbf{orbit equivalent}.

We then have two natural $\LL^p$ versions of orbit equivalence, which are 
direct translations of item (ii) and (iii) from the above equivalence. The 
first one was defined by Austin in \cite{austinBehaviourEntropyBounded2016} and 
the second appears to be new.

\begin{defi}
	Let $p\in[1,+\infty]$. Two measure-preserving actions of finitely generated groups  $\alpha:\Gamma\to\Aut(X,\mu)$ and $\beta:\Lambda\to\Aut(X,\mu)$ are called  \textbf{$\LL^p$-orbit equivalent} if there is $T\in \Aut(X,\mu)$ such that 
	$$T\alpha(\Gamma)T\inv\leq [\beta(\Lambda)]_p\text{ and }T\inv 
	\beta(\Lambda)T\leq[\alpha(\Gamma)]_p.$$
\end{defi}
\begin{rema} For $p=+\infty$, the group 
$[\alpha(\Gamma)]_p$ is  
the group of all $T\in[\alpha(\Gamma)]$ such that the distance from $T(x)$ to 
$x$ is essentially bounded. The $\LL^\infty$ metric on this group is the 
discrete metric, so it is not a Polish group.
\end{rema}

\begin{defi}
	Let $p\in[1,+\infty]$. Two measure-preserving actions of finitely generated groups  $\alpha:\Gamma\to\Aut(X,\mu)$ and $\beta:\Lambda\to\Aut(X,\mu)$ are called  \textbf{$\LL^p$-fully orbit equivalent} if there is $T\in \Aut(X,\mu)$ such that 
	$$T[\alpha(\Gamma)]_pT\inv=[\beta(\Lambda)]_p.$$
\end{defi}

It is clear that $\LL^p$ full orbit equivalence implies $\LL^p$ orbit equivalence, and that $\LL^\infty$ orbit equivalence is equivalent to $\LL^\infty$ full orbit equivalence. Moreover, for all $p<q$ we have that $\LL^q$ orbit equivalence implies $\LL^p$ orbit equivalence, but is is unclear that $\LL^q$ full orbit equivalence implies $\LL^p$ full orbit equivalence. Also note that all these notions collapse to flip-conjugacy for ergodic $\Z$-actions by Belinskaya's theorem. 

A natural intermediate question is to ask what it means for two graphings to share the same $\LL^1$ full group, hoping that the identity map has to be induce an $\LL^\infty$ orbit equivalence between them. However, this is not so, as the following example of Sébastien Gouëzel shows.

\begin{theo}[Gouëzel]\label{thm: not linf equivalent} Given an aperiodic 
$T\in\Aut(X,\mu)$, there is $T'\in\Aut(X,\mu)$ such that $[T]_1=[T']_1$ but 
$T'$ does not belong to the $\LL^\infty$ full group of $T$.
\end{theo}
\begin{proof}
First note that for every $S\in\Aut(X,\mu)$ we have $[STS\inv]_1=S[T]_1S\inv$, so for every $S\in [T]_1$ we have $[STS\inv]_1=[T]_1$. 
So it suffices to find $S\in [T]_1$ such that  $T'=STS\inv$ does not belong to the $\LL^\infty$ full group of $T$. 

Now observe that given any $S\in [T]_1$, the function $x\mapsto d_T(x,STS\inv(x))$ is bounded if and only if the function $x\mapsto d_T(S(x),ST(x))$ is bounded. 
It thus suffices to find $S\in[T]_1$ such that the function $x\mapsto d_T(S(x),ST(x))$ is unbounded. 
This can be achieved by taking $A\subseteq X$ with unbounded return time via Lemma \ref{lem: unbounded return}, and considering the induced transformation $S=T_A$. Indeed if $x\in A$ has return time $n_{T,A}(x)\geq 2$ then $ST(x)=T(x)$ and so $d_T(S(x),ST(x))=n_{T,A}(x)-1$. 
\end{proof}

Gouëzel's example arises as a conjugate of $T$ by some $S\in[T]_1$, but it is natural to ask more generally which $S\in\Aut(X,\mu)$ satisfy $[STS\inv]_1=[T]_1$. Note that such an $S$ has to be an automorphism of the equivalence relation $\mathcal R_T$. Two sources of such $S$ are provided by the group of measure-preserving transformations which conjugate $T$ to $T^{\pm1}$, and elements of the $\LL^1$ full group of $T$. Still, we do not have a good understanding of the group of $S\in\Aut(X,\mu)$ satisfying $[STS\inv]_1=[T^{\pm1}]_1$, and the following question is open.

\begin{ques}Let $T$ be an ergodic measure-preserving transformation. Is the 
group of $S\in\Aut(X,\mu)$ satisfying $[T]_1=[STS\inv]_1$ Borel ? If so, is it 
Polishable ?
\end{ques} 

Note that by Fremlin's reconstruction theorem, every abstract automorphism of $[T]_1$ has to be the conjugacy by some $S\in\Aut(X,\mu)$, so the above question is actually about the automorphism group of $[T]_1$. Also note that the automorphism group of the full group of any measure preserving equivalence relation is Polishable, see \cite[Theorem 6.1]{KechrisGlobalaspectsergodic2010}.

%

\subsection{Symmetric \texorpdfstring{$\LL^p$}{Lp} full groups}\label{sec: symmetric lp full group}

The notions of symmetric (and alternating) topological full groups have natural analogues in our setup. Recall from \cite[Section 5.3]{lemaitremeasurableanaloguesmall2018} that for any $p\in[1,+\infty[$ and any graphing $\Phi$, the $\LL^p$ full group of $\Phi$ is defined as the group of all elements $T$ of the full group of $\mathcal R_\Phi$ such that 
$$\int_Xd_\Phi(x,T(x))^pd\mu(x)<+\infty.$$
It is a Polish group for the metric $d^p_\Phi$ defined by 
$$d^p_\Phi(T_1,T_2)=\left[\int_X d_\Phi(T_1(x),T_2(x))^pd\mu(x)\right]^{1/p},$$ and we  define the following closed normal subgroups.

\begin{defi}
	The \textbf{symmetric $\LL^p$ full group} of a graphing $\Phi$ is the closed subgroup of $[\Phi]_p$ generated by all $n$-cycles for $n\geq 2$. The \textbf{alternating $\LL^p$ full group} is defined as the closed subgroup generated by all $3$-cycles. 
\end{defi}

Note that if $p=1$, both subgroups coincide with the derived $\LL^1$ full group of $\Phi$, and I don't know whether the same holds for $p>1$. Nevertheless,
if $\Phi$ is aperiodic, the same proof as Proposition \ref{prop: derived generated by 3-cycles}	shows that the alternating $\LL^p$ full group of $\Phi$ coincides with its symmetric $\LL^p$ full group. One can also adapt the arguments from \cite{lemaitremeasurableanaloguesmall2018} to show that in the aperiodic case, closed normal subgroups of the symmetric $\LL^p$ full group correspond to invariant sets; in particular the symmetric $\LL^p$ full group of $\Phi$ is topologically simple iff $\Phi$ is ergodic. 
Moreover, the proof of Theorem \ref{thmi: finite topo rank} can easily be adapted to show the following.

\begin{theo} The topological rank of the symmetric $\LL^p$ full group of a 
measure-preserving aperiodic action of a finitely generated group is finite if 
and only if the action has finite Rokhlin entropy.
\end{theo}

In a work in progress with C. Conley and R. Tucker-Drob, we show that for ergodic $\Z$-actions, the symmetric $\LL^p$ full group coincides with the derived $\LL^p$ full group, and that the topological abelianization of the whole $\LL^p$ full group is $\Z$, just as in the $\LL^1$ case. In particular, we can deduce that the whole $\LL^p$ full group of an ergodic measure-preserving transformation has finite topological rank if and only if the transformation has finite entropy, and answer positively question 5.5 from \cite{lemaitremeasurableanaloguesmall2018} in the case the graphing comes from a $\Z$-action.

\subsection{Topological rank of (derived) \texorpdfstring{$\LL^1$}{L1} full groups}

Theorem \ref{thmi: finite topo rank} is not satisfactory in two ways: first, as in the $\Z$ case, a more quantitative statement would be desirable. In particular, establishing a lower bound on the topological rank in terms of the Rokhlin entropy would be very nice. Second, it would be more natural to have a similar statement for the whole $\LL^1$ full group, but in order to do so one has to understand the topological abelianization of $\LL^1$ full groups. We would like to ask the following question, noting that a positive answer would imply that Theorem \ref{thmi: finite topo rank} also holds for the whole $\LL^1$ full group.

\begin{ques}
	Let $\Gamma$ be a finitely generated group. Is the topological abelianization of the $\LL^1$ full group of any ergodic measure-preserving $\Gamma$-action always topologically finitely generated ? 
\end{ques}

At the moment, the question is open even for $\Gamma=\Z^2$. It was answered in 
the affirmative for $\Gamma=\Z$ in \cite{lemaitremeasurableanaloguesmall2018}, 
and with Conley and Tucker-Drob we have a proof that the answer is also yes 
when $\Gamma$ is the infinite dihedral group (we actually show that in this 
case the topological abelianization is always trivial, thus obtaining examples 
of topologically simple $\LL^1$ full groups). Also note that the topological 
rank of the topological abelianization can be greater than the rank of the acting 
group: for a free ergodic $\Z^2$ action, if the two canonical generators have 
diffuse ergodic decomposition then $\R^2$ can be obtained as a quotient of the 
topological abelianization by integrating the cocycles, but it has topological 
rank $3>2$.


\bibliographystyle{cdraifplain}
\bibliography{extracted}

\def\bysame{\leavevmode ---------\thinspace}
\makeatletter\if@francais\providecommand{\og}{<<~}\providecommand{\fg}{~>>}
\else\gdef\og{``}\gdef\fg{''}\fi\makeatother
\def\cdrandname{\&}
\providecommand\cdrnumero{no.~}
\providecommand{\cdredsname}{eds.}
\providecommand{\cdredname}{ed.}
\providecommand{\cdrchapname}{chap.}
\providecommand{\cdrmastersthesisname}{Memoir}
\providecommand{\cdrphdthesisname}{PhD Thesis}
\begin{thebibliography}{10}

\bibitem{austinBehaviourEntropyBounded2016}
{\scshape T.~Austin}, {\og Behaviour of {{Entropy Under Bounded}} and
  {{Integrable Orbit Equivalence}}\fg}, \emph{Geometric and Functional
  Analysis} \textbf{26} (2016), \cdrnumero 6, p.~1483-1525.

\bibitem{Baxterclassergodictransformations1971}
{\scshape J.~R. Baxter}, {\og A class of ergodic transformations having simple
  spectrum\fg}, \emph{Proceedings of the American Mathematical Society}
  \textbf{27} (1971), p.~275-279.

\bibitem{beckerDescriptiveSetTheory1996}
{\scshape H.~Becker {\normalfont \cdrandname}~A.~S. Kechris}, \emph{The
  {{Descriptive Set Theory}} of {{Polish Group Actions}}}, {Cambridge
  University Press}, 1996.

\bibitem{BelinskayaPartitionsLebesguespace1968}
{\scshape R.~M. Belinskaya}, {\og Partitions of {{Lebesgue}} space in
  trajectories defined by ergodic automorphisms\fg}, \emph{Functional Analysis
  and Its Applications} \textbf{2} (1968), \cdrnumero 3, p.~190-199.

\bibitem{CarderiMorePolishfull2016}
{\scshape A.~Carderi {\normalfont \cdrandname}~F.~Le~Ma{\^i}tre}, {\og More
  {{Polish}} full groups\fg}, \emph{Topology and its Applications} \textbf{202}
  (2016), p.~80-105.

\bibitem{connesamenableequivalencerelation1981a}
{\scshape A.~Connes, J.~Feldman {\normalfont \cdrandname}~B.~Weiss}, {\og An
  amenable equivalence relation is generated by a single transformation\fg},
  \emph{Ergodic Theory and Dynamical Systems} \textbf{1} (1981), \cdrnumero 04,
  p.~431-450.

\bibitem{diaconisSpearmanFootruleMeasure1977}
{\scshape P.~Diaconis {\normalfont \cdrandname}~R.~L. Graham}, {\og Spearman's
  {{Footrule}} as a {{Measure}} of {{Disarray}}\fg}, \emph{Journal of the Royal
  Statistical Society. Series B (Methodological)} \textbf{39} (1977),
  \cdrnumero 2, p.~262-268.

\bibitem{dowerkBoundedNormalGeneration2019}
{\scshape P.~Dowerk {\normalfont \cdrandname}~F.~Le~Ma{\^i}tre}, {\og Bounded
  normal generation is not equivalent to topological bounded normal
  generation\fg}, \emph{Extracta Mathematicae} \textbf{34} (2019), \cdrnumero
  1, p.~85-97.

\bibitem{dyeGroupsMeasurePreserving1959}
{\scshape H.~A. Dye}, {\og On {{Groups}} of {{Measure Preserving
  Transformations}}. {{I}}\fg}, \emph{American Journal of Mathematics}
  \textbf{81} (1959), \cdrnumero 1, p.~119-159.

\bibitem{gaboriauOrbitEquivalenceMeasured2011}
{\scshape D.~Gaboriau}, {\og Orbit {{Equivalence}} and {{Measured Group
  Theory}}\fg}, in \emph{Proceedings of the {{International Congress}} of
  {{Mathematicians}} 2010 ({{ICM}} 2010)}, {Published by Hindustan Book Agency
  (HBA), India. WSPC Distribute for All Markets Except in India}, 2011,
  p.~1501-1527.

\bibitem{Giordanoextremelyamenablegroups2007}
{\scshape T.~Giordano {\normalfont \cdrandname}~V.~Pestov}, {\og Some extremely
  amenable groups related to operator algebras and ergodic theory\fg},
  \emph{Journal of the Institute of Mathematics of Jussieu} \textbf{6} (2007),
  \cdrnumero 2, p.~279-315.

\bibitem{GiordanoFullgroupsCantor1999}
{\scshape T.~Giordano, I.~F. Putnam {\normalfont \cdrandname}~C.~F. Skau}, {\og
  Full groups of {{Cantor}} minimal systems\fg}, \emph{Israel Journal of
  Mathematics} \textbf{111} (1999), \cdrnumero 1, p.~285-320.

\bibitem{GlasnerautomorphismgroupGaussian2005}
{\scshape E.~Glasner, B.~Tsirelson {\normalfont \cdrandname}~B.~Weiss}, {\og
  The automorphism group of the {{Gaussian}} measure cannot act pointwise\fg},
  \emph{Israel Journal of Mathematics} \textbf{148} (2005), \cdrnumero 1,
  p.~305-329.

\bibitem{glasnerErgodicTheoryJoinings2003}
{\scshape E.~Glasner}, \emph{Ergodic {{Theory}} via {{Joinings}}}, Mathematical
  {{Surveys}} and {{Monographs}}, vol. 101, {American Mathematical Society},
  {Providence, Rhode Island}, 2003.

\bibitem{jacksonCountableBorelEquivalence2002}
{\scshape S.~Jackson, A.~S. Kechris {\normalfont \cdrandname}~A.~Louveau}, {\og
  Countable borel equivalence relations\fg}, \emph{Journal of Mathematical
  Logic} \textbf{02} (2002), \cdrnumero 01, p.~1-80.

\bibitem{Kacnotionrecurrencediscrete1947}
{\scshape M.~Kac}, {\og On the notion of recurrence in discrete stochastic
  processes\fg}, \emph{Bulletin of the American Mathematical Society}
  \textbf{53} (1947), \cdrnumero 10, p.~1002-1010.

\bibitem{kechrisClassicaldescriptiveset1995}
{\scshape A.~S. Kechris}, \emph{Classical descriptive set theory}, Graduate
  {{Texts}} in {{Mathematics}}, vol. 156, {Springer-Verlag, New York}, 1995.

\bibitem{KechrisGlobalaspectsergodic2010}
\bysame , \emph{Global aspects of ergodic group actions}, Mathematical
  {{Surveys}} and {{Monographs}}, vol. 160, {American Mathematical Society,
  Providence, RI}, 2010.

\bibitem{kittrellTopologicalPropertiesFull2010}
{\scshape J.~Kittrell {\normalfont \cdrandname}~T.~Tsankov}, {\og Topological
  properties of full groups\fg}, \emph{Ergodic Theory and Dynamical Systems}
  \textbf{30} (2010), \cdrnumero 2, p.~525-545.

\bibitem{lemaitreGroupesPleinsPreservant2014}
{\scshape F.~Le~Ma{\^i}tre}, {\og {Sur les groupes pleins pr\'eservant une
  mesure de probabilit\'e}\fg}, \cdrphdthesisname, ENS Lyon, 2014.

\bibitem{lemaitrefullgroupsnonergodic2016}
\bysame , {\og On full groups of non-ergodic probability-measure-preserving
  equivalence relations\fg}, \emph{Ergodic Theory and Dynamical Systems}
  \textbf{36} (2016), \cdrnumero 7, p.~2218-2245.

\bibitem{lemaitremeasurableanaloguesmall2018}
\bysame , {\og On a measurable analogue of small topological full groups\fg},
  \emph{Advances in Mathematics} \textbf{332} (2018), p.~235-286.

\bibitem{marksTopologicalGeneratorsFull2016}
{\scshape A.~S. Marks}, {\og Topological generators for full groups of
  hyperfinite pmp equivalence relations\fg}, \emph{arXiv:1606.08080 [math]}
  (2016).

\bibitem{Matuiremarkstopologicalfull2006}
{\scshape H.~Matui}, {\og Some remarks on topological full groups of {{Cantor}}
  minimal systems\fg}, \emph{International Journal of Mathematics} \textbf{17}
  (2006), \cdrnumero 02, p.~231-251.

\bibitem{matuiTopologicalFullGroups2013}
\bysame , {\og Topological full groups of one-sided shifts of finite type\fg},
  \emph{Journal f\"ur die reine und angewandte Mathematik (Crelles Journal)}
  \textbf{2015} (2013), \cdrnumero 705, p.~35-84.

\bibitem{nekrashevychSimplegroupsdynamical2017}
{\scshape V.~Nekrashevych}, {\og Simple groups of dynamical origin\fg},
  \emph{Ergodic Theory and Dynamical Systems} (2017), p.~1-26.

\bibitem{NekrashevychPalindromicsubshiftssimple2018}
{\scshape V.~Nekrashevych}, {\og Palindromic subshifts and simple periodic
  groups of intermediate growth\fg}, \emph{Annals of Mathematics. Second
  Series} \textbf{187} (2018), \cdrnumero 3, p.~667-719.

\bibitem{PestovConcentrationmeasurewhirly2009}
{\scshape V.~Pestov}, {\og Concentration of measure and whirly actions of
  {{Polish}} groups\fg}, in \emph{Proceedings of the 1st {{Mathematical
  Society}} of {{Japan Seasonal Institute}} "{{Probabilistic Approach}} to
  {{Geometry}}" ({{Kyoto}}, {{July}}-{{Aug}}. 2008)}, Advanced {{Studies}} in
  {{Pure Mathematics}}, vol.~57, {Mathematical Society of Japan}, 2009,
  p.~383-403.

\bibitem{rokhlinLecturesentropytheory1967}
{\scshape V.~A. Rokhlin}, {\og Lectures on the entropy theory of
  measure-preserving transformations\fg}, \emph{Russian Mathematical Surveys}
  \textbf{22} (1967), \cdrnumero 5, p.~1.

\bibitem{rosendalCoarseGeometryTopological2018}
{\scshape C.~Rosendal}, \emph{Coarse {{Geometry}} of {{Topological Groups}}},
  Preliminary Version, 2018.

\bibitem{sewardErgodicActionsCountable2015}
{\scshape B.~Seward}, {\og Ergodic actions of countable groups and finite
  generating partitions\fg}, \emph{Groups, Geometry, and Dynamics} \textbf{9}
  (2015), \cdrnumero 3, p.~793-810.

\bibitem{sewardKriegerfinitegenerator2019}
\bysame , {\og Krieger's finite generator theorem for actions of countable
  groups {{I}}\fg}, \emph{Inventiones mathematicae} \textbf{215} (2019),
  \cdrnumero 1, p.~265-310.

\end{thebibliography}

\end{document}